\documentclass[leqno, 11pt]{amsart}

\usepackage[top=90pt, bottom=70pt, left=80pt, right=80pt]{geometry}
\usepackage{graphicx}
\usepackage[latin1]{inputenc}
\usepackage[catalan,english]{babel}
\usepackage{amsmath,amssymb}
\usepackage{amsthm,enumerate}
\usepackage{latexsym}
\usepackage{xypic}
\usepackage[all]{xy}
\usepackage{pxfonts}
\usepackage{mathrsfs}
\usepackage{multirow}

\newtheorem{thm}{Theorem}[section]
\newtheorem{defi}[thm]{Definition}

\newtheorem{prop}[thm]{Proposition}
\newtheorem{lmm}[thm]{Lemma}
\newtheorem{cor}[thm]{Corollary}

\newtheorem{thml}{Theorem}

\theoremstyle{definition}
\newtheorem{rmk}[thm]{Remark}{\rm}

\theoremstyle{definition}
{\rm}

\newcommand{\F}{\mathbb{F}}
\newcommand{\Z}{\mathbb{Z}}

\newcommand{\N}{\mathbb{N}}

\newcommand{\gen}[1]{\langle #1 \rangle}
\newcommand{\fusion}{\mathcal{F}}

\newcommand{\linking}{\mathcal{L}}

\newcommand{\transp}{\mathcal{T}}
\newcommand{\g}{\mathcal{G}}
\newcommand{\ploc}{(S, \fusion, \linking)}
\newcommand{\pp}{\mathcal{P}}
\newcommand{\hh}{\mathcal{H}}
\newcommand{\mm}{\mathcal{M}}

\newcommand{\hatmm}{\widehat{\mathcal{M}}}
\newcommand{\conj}[2]{{#1}^{#2}}

\newcommand{\dihed}{D_{2^{\infty}}}
\newcommand{\quat}{Q_{2^{\infty}}}

\newcommand{\padic}{\Z^\wedge_p}
\newcommand{\adic}{\Z^\wedge_2}
\newcommand{\ptor}{\Z/p^{\infty}}

\newcommand{\functor}{(\underline{\phantom{B}})^{\bullet}}

\begin{document}

\title{Unstable Adams operations acting on $p$-local compact groups and fixed points}
\author{A. Gonz\'{a}lez}

\begin{abstract}

We prove in this paper that every $p$-local compact group is approximated by transporter systems over finite $p$-groups. To do so, we use unstable Adams operations acting on a given $p$-local compact group and study the structure of resulting fixed points.

\end{abstract}

\maketitle

The theory of $p$-local compact groups was introduced by C. Broto, R. Levi and B. Oliver in \cite{BLO3} as the natural generalization of $p$-local finite groups to include some infinite structures, such as compact Lie groups or $p$-compact groups, in an attempt to give categorical models for a larger class of $p$-completed classifying spaces.

Nevertheless, when passing from a finite setting to an infinite one, some of the techniques used in the former case are not available any more. As a result, some of the more important results in \cite{BLO2} were not extended to $p$-local compact groups, and, roughly speaking, that $p$-local compact groups are not yet as well understood as $p$-local finite groups. It is then the aim of this paper to shed some light on the new theory introduced in \cite{BLO3}.

The underlying idea of this paper can be traced back to work of E. M. Friedlander and G. Mislin, \cite{F1}, \cite{F2}, \cite{FM1} and \cite{FM2}, where the authors use unstable Adams operations (or Frobenius maps in the algebraic setting) to approximate classifying spaces of compact Lie groups by classifying spaces of finite groups. More recently, C. Broto and J. M. M{\o}ller  studied a similar construction for connected $p$-compact groups in \cite{BM}.

Here, by an approximation of a compact Lie group $G$ by finite groups we mean the existence of a locally finite group $G^{\circ}$, together with a mod $p$ homotopy equivalence $BG^{\circ} \to BG$ (that is, this map induces a mod $p$ homology isomorphism). Since $G^{\circ}$ is locally finite, it can be described as a colimit of finite groups, and this allows then to extend known properties of finite groups to compact Lie groups. Of course, this argument is unnecessary in the classical setting of compact Lie groups, since other techniques are at hand.

As we have mentioned, the works of Friedlander and Mislin depended on Frobenius maps and their analogues in topological $K$-theory, unstable Adams operations. For $p$-local compact groups, unstable Adams operations were constructed for all $p$-local compact groups in the doctoral thesis \cite{Junod}, although in this work we will use the more refined version of unstable Adams operations from \cite{JLL}.

One can then study the action of unstable Adams operations on $p$-local compact groups from a categorical point of view, focusing on the definition of a $p$-local compact group as a triple $\ploc$, which is in fact the approach that we have adopted in this paper, and which will lead to a rather explicit description of the fixed points of a $p$-local compact group under the action of an unstable Adams operation (of high enough degree).

The following result, which is the main theorem in this paper, will be restated and proved as Theorems \ref{transpi} and \ref{hocolim1}.

\begin{thml}\label{ThA}

Let $\g = \ploc$ be a $p$-local compact group, and let $\Psi$ be an unstable Adams operation on $\g$. Then, $\Psi$ defines a family of finite transporter systems $\{\linking_i\}_{i \in \N}$, together with faithful functors $\Theta_i: \linking_i \to \linking_{i+1}$ for all $i$, such that there is a mod $p$ homotopy equivalence
$$
B\g \simeq_p \operatorname{hocolim} |\linking_i|.
$$

\end{thml}

Each transporter system $\linking_i$ is associated to an underlying fusion system $\fusion_i$, which is $Ob(\linking_i)$-generated and $Ob(\linking_i)$-saturated (see definition \ref{hgenhsat}). The notation comes from \cite{BCGLO1}. The first property (generation) means that morphisms in $\fusion_i$ are compositions of restrictions of morphisms among the objects in $Ob(\linking_i)$, and the second property (saturation) means that the objects in $Ob(\linking_i)$ satisfy the saturation axioms.

The saturation of the fusion systems $\fusion_i$ remains unsolved in the general case, but we study some examples in this paper where, independently of the operation $\Psi$, the triples $\g_i$ are always (eventually) $p$-local finite groups. To simplify the statements below, we will say that $\Psi$ induces an \textit{approximation of $\g$ by $p$-local finite groups} if the triples $\g_i$ in Theorem A are $p$-local finite groups. This definition will be made precise in section \S 3.

The following results correspond to Theorems \ref{rank1} and \ref{Un} respectively.

\begin{thml}

Let $\g$ be a $p$-local compact group of rank $1$, and let $\Psi$ be an unstable Adams operation acting on $\g$. Then, $\Psi$ induces an approximation of $\g$ by $p$-local finite groups.

\end{thml}

\begin{thml}

Let $\g$ be the $p$-local compact group induced by the compact Lie group $U(n)$, and let $\Psi$ be an unstable Adams operation acting on $\g$. Then, $\Psi$ induces an approximation of $\g$ by $p$-local finite groups.

\end{thml}

As an immediate consequence of approximations of $p$-local compact groups by $p$-local finite groups, we prove in section \S 3 a Stable Elements Theorem for $p$-local compact groups (whenever such an approximation is available). Stable Elements Theorem has proved to be a rather powerful tool in the study of $p$-local finite groups, and one would of course like to have a general proof in the compact case. In this sense, our conjecture is that the constructions that we introduce in this paper yield approximations by $p$-local finite groups for all $p$-local compact groups.

One could also choose a different approach to the study of fixed points in this $p$-local setting. Indeed, given an unstable Adams operation $\Psi$ acting on $\g$, one could consider the homotopy fixed points of $B\g$ under the natural map induced by an unstable Adams operation acting on $\g$, namely the homotopy pull-back
$$
\xymatrix@C=15mm{
X \ar[r] \ar[d] & B\g \ar[d]^{\Delta} \\
B\g \ar[r]_{(id, |\Psi|)} & B\g \times B\g,\\
}
$$
and apply the topological tools provided in \cite{BLO2} and \cite{BLO4} to study the homotopy type of $X$. This point of view is in fact closer to the work of Broto and M{\o}ller \cite{BM} mentioned above, and will constitute the main subject in a sequel of this paper, where we will relate the constructions introduced in this paper and homotopy fixed points.

The paper is organized as follows. The first section contains the main definitions of discrete $p$-toral groups, (saturated) fusion systems, centric linking systems, transporter systems and $p$-local compact groups. This section also contains the definition of unstable Adams operations from \cite{JLL}. In the second section we study the effect of a whole family of unstable Adams operations acting on a fixed $p$-local compact group. This section is to be considered as a set of tools that we use in the following section. Indeed, the third section contains the construction of the triples $\g_i$ and the proof for Theorem A above. It also contains a little discussion about approximations of $p$-local compact groups by $p$-local finite groups, where we prove a Stable Elements Theorem (in this particular situation) for $p$-local compact groups. The last section is devoted to examples (Theorems B and C above).

The author would like to thank professors C. Broto, R. Levi and B. Oliver for many interesting conversations and discussions on the contents of this paper, as well as for their help and encouragement. Also many thanks to A. Libman for (at least one) thorough early review on this paper, as well as many useful hints. Finally, the author would like to thank the Universitat Aut{\`o}noma de Barcelona and the University of Aberdeen, where most of this work has been carried on. This work summarizes part of the doctoral thesis of the author, realized under the supervision of C. Broto and R. Levi, and as such the author would like to thank both of them for their continuous help.

The author was partially supported by FEDER/MEC grant BES-2005-11029.


\section{Background on $p$-local compact groups} 

In this section, we review the definition of a $p$-local compact group and state some results that we will use later on. Mostly, the contents in this section come from \cite{BLO3}. When this is the case, we will provide a reference where the reader can find a proof,  in order to simplify the exposition of this paper.


\subsection{Discrete $p$-toral groups and fusion systems}

\begin{defi}\label{defidiscreteptoral and fusion systems}

A \textbf{discrete $p$-torus} is a group $T$ isomorphic to a finite direct product of copies of $\ptor$.

A \textbf{discrete $p$-toral group} $P$ is an extension of a finite $p$-group $\pi$ by a discrete $p$-torus $T$. For such a group, we call $T \cong (\ptor)^r$ the \textbf{maximal torus} of $P$, and define the \textbf{rank} of $P$ as $r$.

\end{defi}

Discrete $p$-toral groups were characterized in \cite{BLO3} (Proposition 1.2) as those groups satisfying the descending chain condition and such that every finitely generated subgroup is a finite $p$-group.

In this paper we will deal with some infinite groups. For an infinite group $G$, we say that $G$ \textit{has Sylow $p$-subgroups} if $G$ contains a discrete $p$-toral group $S$ such that any finite $p$-subgroup of $G$ is $G$-conjugate to a subgroup of $S$.

For a group $G$ and subgroups $P,P' \leq G$, define $N_G(P,P') = \{g \in G \mbox{ } | \mbox{ } g \cdot P \cdot g^{-1} \leq P'\}$ and $Hom_G(P,P') = N_G(P,P')/C_G(P)$. Fusion systems over discrete $p$-toral groups are defined just as they were defined in the finite case.

\begin{defi}\label{defifusion}

A \textbf{fusion system} $\fusion$ over a discrete $p$-toral group $S$ is a category whose objects are the subgroups of $S$ and whose morphism sets $Hom_{\fusion}(P,P')$ satisfy the following conditions:
\begin{enumerate}[(i)]
\item $Hom_S(P,P') \subseteq Hom_{\fusion}(P,P') \subseteq Inj(P,P')$ for all $P,P' \leq S$.
\item Every morphism in $\fusion$ factors as an isomorphism in $\fusion$ followed by an inclusion.
\end{enumerate}

Given a fusion system $\fusion$ over a discrete $p$-toral group $S$, we will often refer to $T$ also as the \textbf{maximal torus} of $\fusion$, and the \textbf{rank} of $\fusion$ will then be the rank of the discrete $p$-toral group $S$. Two subgroups $P,P'$ are called \textbf{$\fusion$-conjugate} if $Iso_{\fusion}(P,P') \neq \emptyset$. For a subgroup $P \leq S$, we denote 
$$
\conj{P}{\fusion} = \{P' \leq S \mbox{ } | \mbox{ } P'  \mbox{ is } \fusion \mbox{-conjugate to } P\}.
$$

\end{defi}

For a discrete $p$-toral group $P$, the order of $P$ was defined in \cite{BLO3} as $|P| = (rk(P), |P/T_P|)$, where $T_P$ is the maximal torus of $P$. Thus, given two discrete $p$-toral groups $P$ and $Q$ we say that $|P| \leq |Q|$ if either $rk(P) < rk(Q)$, or $rk(P) = rk(Q)$ and $|P/T_P| \leq |Q/T_Q|$.

\begin{defi}\label{defisaturation}

Let $\fusion$ be a fusion system over a discrete $p$-toral group $S$. A subgroup $P \leq S$ is called \textbf{fully $\fusion$-normalized}, resp. \textbf{fully $\fusion$-centralized}, if $|N_S(P')| \leq |N_S(P)|$, resp. $|C_S(P')| \leq |C_S(P)|$, for all $P' \leq S$ which is $\fusion$-conjugate to $P$.

The fusion system $\fusion$ is called \textbf{saturated} if the following three conditions hold:
\begin{enumerate}[(I)]
\item For each $P \leq S$ which is fully $\fusion$-normalized, $P$ is fully $\fusion$-centralized, $Out_{\fusion}(P)$ is finite and $Out_S(P) \in Syl_p(Out_{\fusion}(P))$.
\item If $P \leq S$ and $f \in Hom_{\fusion}(P,S)$ is such that $P' = f(P)$ is fully $\fusion$-centralized, then there exists $\widetilde{f} \in Hom_{\fusion}(N_f,S)$ such that $f = \widetilde{f}_{|P}$, where
$$
N_f = \{g \in N_S(P) | f \circ c_g \circ f^{-1} \in Aut_S(P')\}.
$$
\item If $P_1 \leq P_2 \leq P_3 \leq \ldots$ is an increasing sequence of subgroups of $S$, with $P = \cup_{n=1}^{\infty} P_n$, and if $f \in Hom(P,S)$ is any homomorphism such that $f_{|P_n} \in Hom_{\fusion}(P_n,S)$ for all $n$, then $f \in Hom_{\fusion}(P,S)$.
\end{enumerate}

\end{defi}

Let $(S, \fusion)$ be a saturated fusion system over a discrete $p$-toral group. Note that, by definition, all the automorphism groups in a saturated fusion system are artinian and locally finite. The condition in axiom (I) of $Out_{\fusion}(P)$ being finite is in fact redundant, as was pointed out in Lemmas 2.3 and  2.5 \cite{BLO3}, where the authors show that the set $Rep_{\fusion}(P,Q) = Inn(Q) \setminus Hom_{\fusion}(P,Q)$ is finite for all $P, Q \in Ob(\fusion)$.

Given a discrete $p$-toral group $S$ and a subgroup $P \leq S$, we say that $P$ is \textit{centric in $S$}, or $S$-centric, if $C_S(P) = Z(P)$. We next define $\fusion$-centric and $\fusion$-radical subgroups.

\begin{defi}\label{deficentricrad}

Let $\fusion$ be a saturated fusion system over a discrete $p$-toral group. A subgroup $P \leq S$ is called \textbf{$\fusion$-centric} if all the elements of $\conj{P}{\fusion}$ are centric in $S$:
$$
C_S(P') = Z(P') \mbox{ for all } P' \in \conj{P}{\fusion}.
$$

A subgroup $P \leq S$ is called \textbf{$\fusion$-radical} if $Out_{\fusion}(P)$ contains no nontrivial normal $p$-subgroup:
$$
O_p(Out_{\fusion}(P)) = \{1\}.
$$

\end{defi}

Clearly, $\fusion$-centric subgroups are fully $\fusion$-centralized, and conversely, if $P$ is fully $\fusion$-centralized and centric in $S$, then it is $\fusion$-centric.

There is, of course, a big difference between working with finite $p$-groups and with discrete $p$-toral groups: the number of conjugacy classes of subgroups. Fortunately, in \cite{BLO3} the authors came out with a way of getting rid of infinitely many conjugacy classes while keeping the structure of a given fusion system.

This construction will be rather important in this paper, and we reproduce it here for the sake of a better reading. Let then $(S, \fusion)$ be a saturated fusion system over a discrete $p$-toral group, and let $T$ be the maximal torus of $\fusion$ and $W = Aut_{\fusion}(T)$.

\begin{defi}\label{defibullet}

Let $\fusion$ be a saturated fusion system over a discrete $p$-toral group $S$, and let $e$ denote the exponent of $S/T$,
$$
e = exp(S/T) = min\{p^k \mbox{ } | \mbox{ } x^{p^k} \in T \mbox{ for all } x \in S\}.
$$

\begin{enumerate}[(i)]

\item For each $P \leq T$, let
$$
I(P) = \{t \in T \mbox{ } | \mbox{ } \omega(t) = t \mbox{ for all } \omega \in W \mbox{ such that } \omega_{|P} = id_P\},
$$
and let $I(P)_0$ denote its maximal torus.
\item For each $P \leq S$, set $P^{[e]} = \{x^{p^e} \mbox{ } | \mbox{ } x \in P\} \leq T$, and let
$$
P^{\bullet} = P \cdot I(P^{[e]})_0 = \{xt | x \in P, t \in I(P^{[e]})_0\}.
$$
\item Set $\hh^{\bullet} = \{P^{\bullet} | P \in \fusion\}$ and let $\fusion^{\bullet}$ be the full subcategory of $\fusion$ with object set $Ob(\fusion^{\bullet}) = \hh^{\bullet}$.

\end{enumerate}

\end{defi}

The following is a summary of section \S 3 in \cite{BLO3}.

\begin{prop}\label{bulletprop}

Let $\fusion$ be a saturated fusion system over a discrete $p$-toral group $S$. Then,

\begin{enumerate}[(i)]

\item the set $\hh^{\bullet}$ contains finitely many $S$-conjugacy classes of subgroups of $S$; and

\item every morphism $(f:P \to Q)  \in Mor(\fusion)$ extends uniquely to a morphism $f^{\bullet}: P^{\bullet} \to Q^{\bullet}$.

\end{enumerate}

This makes $\functor: \fusion \to \fusion^{\bullet}$ into a functor. This functor is an idempotent functor ($(P^{\bullet})^{\bullet} = P^{\bullet}$), carries inclusions to inclusions ($P^{\bullet} \leq Q^{\bullet}$ whenever $P \leq Q$), and is left adjoint to the inclusion $\fusion^{\bullet} \subseteq \fusion$.

\end{prop}

Finally, we state Alperin's fusion theorem for saturated fusion systems over discrete $p$-toral groups.

\begin{thm}\label{Alperin}

(3.6 \cite{BLO3}). Let $\fusion$ be a saturated fusion system over a discrete $p$-toral group $S$. Then, for each $f \in Iso_{\fusion}(P,P')$ there exist sequences of subgroups of $S$
$$
\begin{array}{ccc}
P = P_0, P_1, \ldots, P_k = P' & \mbox{and} & Q_1, \ldots, Q_k,\\
\end{array}
$$
and elements $f_j \in Aut_{\fusion}(Q_j)$ such that
\begin{enumerate}[(i)]
\item for each $j$, $Q_j$ is fully normalized in $\fusion$, $\fusion$-centric and $\fusion$-radical;
\item also for each $j$, $P_{j-1}, P_j \leq Q_j$ and $f_j(P_{j-1}) = P_j$; and
\item $f = f_k \circ f_{k-1} \circ \ldots \circ f_1$.

\end{enumerate}

\end{thm}

It is also worth mentioning the alternative set of saturation axioms provided in \cite{KS}, since it will be useful in later sections. Let $\fusion$ be a fusion system over a finite $p$-group $S$, and consider the following conditions:
\begin{enumerate}[(I')]
\item $Out_S(S) \in Syl_p(Out_{\fusion}(S))$.
\item Let $f:P \to S$ be a morphism in $\fusion$ such that $P' = f(P)$ is fully $\fusion$-normalized. Then, $f$ extends to a morphism $\widetilde{f}: N_f \to S$ in $\fusion$, where
$$
N_f = \{g \in N_S(P)|f \circ c_g \circ f^{-1} \in Aut_S(P')\}.
$$
\end{enumerate}
The following result is a compendium of appendix A in \cite{KS}.

\begin{prop}\label{Stancuaxioms}

Let $\fusion$ be a fusion system over a finite $p$-group $S$. Then, $\fusion$ is saturated (in the sense of definition \ref{defisaturation}) if and only if it satisfies axioms (I') and (II') above.

\end{prop}

A more general version of this result for fusion systems over discrete $p$-toral groups was proved in \cite{Gonza}, but is of no use in the paper.


\subsection{Linking systems and transporter systems}

Linking systems are the third and last ingredient needed to form a $p$-local compact group.

\begin{defi}\label{defilinking}

Let $\fusion$ be a saturated fusion system over a discrete $p$-toral group $S$. A \textbf{centric linking system associated to} $\fusion$ is a category $\linking$ whose objects are the $\fusion$-centric subgroups of $S$, together with a functor
$$
\rho: \linking \longrightarrow \fusion^c
$$
and ``distinguished'' monomorphisms $\delta_P: P \to Aut_{\linking}(P)$ for each $\fusion$-centric subgroup $P \leq S$, which satisfy the following conditions.

\begin{enumerate}[(A)]
\item $\rho$ is the identity on objects and surjective on morphisms. More precisely, for each pair of objects $P, P' \in \linking$, $Z(P)$ acts freely on $Mor_{\linking}(P,P')$ by composition (upon identifying $Z(P)$ with $\delta_P(Z(P)) \leq Aut_{\linking}(P)$), and $\rho$ induces a bijection
$$
Mor_{\linking}(P,P')/Z(P) \stackrel{\cong} \longrightarrow Hom_{\fusion}(P,P').
$$
\item For each $\fusion$-centric subgroup $P \leq S$ and each $g \in P$, $\rho$ sends $\delta_P(g) \in Aut_{\linking}(P)$ to $c_g \in Aut_{\fusion}(P)$.
\item For each $\varphi \in Mor_{\linking}(P,P')$ and each $g \in P$, the following square commutes in $\linking$:
$$
\xymatrix{
P \ar[r]^{\varphi} \ar[d]_{\delta_P(g)} & P' \ar[d]^{\delta_{P'}(h)} \\
P \ar[r]_{\varphi} & P',
}
$$
where $h = \rho(\varphi)(g)$.

\end{enumerate}

A \textbf{$p$-local compact group} is a triple $\g = (S, \fusion, \linking)$, where $S$ is a discrete $p$-toral group, $\fusion$ is a saturated fusion system over $S$, and $\linking$ is a centric linking system associated to $\fusion$. The \textbf{classifying space} of $\g$ is the $p$-completed nerve
$$
B\g \stackrel{def} = |\linking|^{\wedge}_p.
$$
Given a $p$-local compact group $\g$, the subgroup $T \leq S$ will be called the \textbf{maximal torus} of $\g$, and the \textbf{rank} of $\g$ will then be the rank of the discrete $p$-toral group $S$.

\end{defi}

We will in general denote a $p$-local compact group just by $\g$, assuming that $S$ is its Sylow $p$-subgroup, $\fusion$ is the corresponding fusion system, and $\linking$ is the corresponding linking system.

As expected, the classifying space of a $p$-local compact group behaves ``nicely'', meaning that $B\g = |\linking|^{\wedge}_p$ is a $p$-complete space (in the sense of \cite{BK}) whose fundamental group is a finite $p$-group, as proved in Proposition 4.4 \cite{BLO3}.

Next we state some properties of linking systems. We start with an extended version of Lemma 4.3 \cite{BLO3}.

\begin{lmm}\label{3.2OV}

Let $\g$ be a $p$-local compact group. Then, the following holds.

\begin{enumerate}[(i)]

\item Fix morphisms $f \in Hom_{\fusion}(P,Q)$ and $f' \in Hom_{\fusion}(Q,R)$, where $P,Q,R \in \linking$. Then, for any pair of liftings $\varphi '\in \rho^{-1}_{Q,R}(f')$ and $\omega \in \rho^{-1}_{P,R}(f' \circ f)$, there is a unique lifting $\varphi \in \rho^{-1}_{P,Q}(f)$ such that $\varphi ' \circ \varphi = \omega$.

\item All morphisms in $\linking$ are monomorphisms in the categorical sense. That is, for all $P,Q,R \in \linking$ and all $\varphi_1, \varphi_2 \in Mor_{\linking}(P,Q)$, $\psi \in Mor_{\linking}(Q,R)$, if $\psi \circ \varphi_1 = \psi \circ \varphi_2$ then $\varphi_1 = \varphi_2$.

\item For every morphism $\varphi \in Mor_{\linking}(P,Q)$ and every $P_0, Q_0 \in \linking$ such that $P_0 \leq P$, $Q_0 \leq Q$ and $\rho(\varphi)(P_0) \leq Q_0$, there is a unique morphism $\varphi_0 \in Mor_{\linking}(P_0, Q_0)$ such that $\varphi \circ \iota_{P_0,P} = \iota_{Q_0,Q} \circ \varphi_0$. In particular, every morphism in $\linking$ is a composite of an isomorphism followed by an inclusion.

\item All morphisms in $\linking$ are epimorphisms in the categorical sense. In other words, for all $P,Q,R \in \linking$ and all $\varphi \in Mor_{\linking}(P,Q)$ and $\psi_1, \psi_2 \in Mor_{\linking}(Q,R)$, if $\psi_1 \circ \varphi = \psi_2 \circ \varphi$ then $\psi_1 = \psi_2$.

\end{enumerate}

\end{lmm}

\begin{proof}

Since the functor $\rho: \linking \to \fusion^c$ is both source and target regular (in the sense of definition A.5 \cite{OV}) by axiom (A) of linking systems, the proof for Lemma 3.2 \cite{OV} applies in this case as well.

\end{proof}

Let $\g$ be a $p$-local compact group, and, for each $P \in \linking$ fix a lifting of $incl^S_P:P \to S$ in $\linking$, $\iota_{P,S} \in Mor_{\linking}(P,S)$. Then, by the above Lemma, we may complete this to a family of inclusions $\{\iota_{P,P'}\}$ in a unique way and such that $\iota_{P,S} =\iota_{P',S} \circ \iota_{P,P'}$ whenever it makes sense.

\begin{lmm}\label{1.11BLO2}

Fix such a family of inclusions $\{\iota_{P,P'}\}$ in $\linking$. Then, for each $P,P' \in \linking$, there are unique injections
$$
\delta_{P,P'}:N_S(P,P') \longrightarrow Mor_{\linking}(P,P')
$$
such that
\begin{enumerate}[(i)]
\item $\iota_{P',S} \circ \delta_{P,P'}(g) = \delta_S(g) \circ \iota_{P,S}$, for all $g \in N_S(P,P')$, and
\item $\delta_P$ is the restriction to $P$ of $\delta_{P,P}$.
\end{enumerate}

\end{lmm}

\begin{proof}

The proof for Proposition 1.11 \cite{BLO2} applies here as well, using Lemma \ref{3.2OV} (i) above instead of Proposition 1.10 (a) \cite{BLO2}.

\end{proof}

We now introduce transporter systems. The notion that we present here was first used in \cite{OV} as a tool to study certain extensions of $p$-local finite groups. In this sense, most of the results in \cite{OV} can be extended to the compact case without restriction, as proved in \cite{Gonza}, but we will not make use of such results in this paper. More details can be found at \cite{Gonza}.

Let $G$ be an artinian locally finite group with Sylow $p$-subgroups, and fix $S \in Syl_p(G)$. We define $\transp_S(G)$ as the category whose object set is $Ob(\transp_S(G)) = \{P \leq S\}$, and such that
$$
Mor_{\transp_S(G)}(P,P') = N_G(P,P') = \{g \in G \mbox{ } | \mbox{ } gPg^{-1} \leq P'\}.
$$
For a subset $\hh \subseteq Ob(\transp_S(G))$, $\transp_{\hh}(G)$ denotes the full subcategory of $\transp_S(G)$ with object set $\hh$.

\begin{defi}\label{defitransporter}

Let $\fusion$ be a fusion system over a discrete $p$-toral group $S$. A \textbf{transporter system} associated to $\fusion$ is a category $\transp$ such that
\begin{enumerate}[(i)]
\item $Ob(\transp) \subseteq Ob(\fusion)$;
\item for all $P \in Ob(\transp)$, $Aut_{\transp}(P)$ is an artinian locally finite group;
\end{enumerate}
together with a couple of functors
$$
\transp_{Ob(\transp)}(S) \stackrel{\varepsilon} \longrightarrow \transp \stackrel{\rho} \longrightarrow \fusion,
$$
satisfying the following axioms:
\begin{itemize}
 \item[(A1)] $Ob(\transp)$ is closed under $\fusion$-conjugacy and overgroups. Also, $\varepsilon$ is the identity on objects and $\rho$ is inclusion on objects.

 \item[(A2)] For each $P \in Ob(\transp)$, let
$$
E(P) = Ker(Aut_{\transp}(P) \to Aut_{\fusion}(P)).
$$
Then, for each $P, P' \in Ob(\transp)$, $E(P)$ acts freely on $Mor_{\transp}(P, P')$ by right composition, and $\rho_{P, P'}$ is the orbit map for this action. Also, $E(P')$ acts freely on $Mor_{\transp}(P, P')$ by left composition.

 \item[(B)] For each $P, P' \in Ob(\transp)$, $\varepsilon_{P,P'}: N_S(P,P') \to Mor_{\transp}(P, P')$ is injective, and the composite $\rho_{P,P'} \circ \varepsilon_{P,P'}$ sends $g \in N_S(P, P')$ to $c_g \in Hom_{\fusion}(P,P')$.

 \item[(C)] For all $\varphi \in Mor_{\transp}(P,P')$ and all $g \in P$, the following diagram commutes in $\transp$:
$$
\xymatrix{
P \ar[r]^{\varphi} \ar[d]_{\varepsilon_{P,P}(g)} & P' \ar[d]^{\varepsilon_{P',P'}(\rho(\varphi)(g))} \\
P \ar[r]_{\varphi} & P'.
}
$$

 \item[(I)] $Aut_{\transp}(S)$ has Sylow $p$-subgroups, and $\varepsilon_{S,S}(S) \in Syl_p(Aut_{\transp}(S))$.

 \item[(II)] Let $\varphi \in Iso_{\transp}(P,P')$, and $P \lhd R \leq S$, $P' \lhd R' \leq S$ such that
$$
\varphi \circ \varepsilon_{P,P}(R) \circ \varphi^{-1} \leq \varepsilon_{P',P'}(R').
$$
Then, there is some $\widetilde{\varphi} \in Mor_{\transp}(R, R')$ such that $\widetilde{\varphi} \circ \varepsilon_{P,R}(1) = \varepsilon_{P', R'}(1) \circ \varphi$, that is, the following diagram is commutative in $\transp$:
$$
\xymatrix{
P \ar[r]^{\varphi} \ar[d]_{\varepsilon_{P,R}(1)} & P' \ar[d]^{\varepsilon_{P',R'}(1)} \\
R \ar[r]_{\widetilde{\varepsilon}} & R'.
}
$$

 \item[(III)] Let $P_1 \leq P_2 \leq \ldots$ be an increasing sequence of subgroups in $Ob(\transp)$, and $P = \bigcup_{n=1}^{\infty} P_n$. Suppose in addition that there exists $\psi_n \in Mor_{\transp}(P_n, S)$ such that
$$
\psi_n = \psi_{n+1} \circ \varepsilon_{P_n, P_{n+1}}(1)
$$
for all $n$. Then, there exists $\psi \in Mor_{\transp}(P, S)$ such that $\psi_n = \psi \circ \varepsilon_{P_n, P}(1)$ for all $n$.
\end{itemize}

Given a transporter system $\transp$, the \textbf{classifying space} of $\transp$ is the space $B\transp \stackrel{def} = |\transp|^{\wedge}_p$.

\end{defi}

Note that, in axiom (III), $P$ is an object in $Ob(\transp)$, since $Ob(\transp)$ is closed under $\fusion$-conjugacy and overgroups. As in \cite{OV}, the axioms are labelled to show their relation with the axioms for linking and fusion systems respectively. Note also that, whenever $S$ is a finite $p$-group, the above definition agrees with that in \cite{OV}.

\begin{prop}\label{3.5OV}

Let $\g = (S,\fusion, \linking)$ be a $p$-local compact group. Then, $\linking$ is a transporter system associated to $\fusion$.

\end{prop}

\begin{proof}

The usual projection functor $\rho: \linking \to \fusion$ in the definition of a linking system plays also the role of the projection functor in the definition of transporter system. Also, in Lemma \ref{1.11BLO2} we have defined a functor $\varepsilon: \transp_{Ob(\linking)}(S) \to \linking$. It remains to check that $\linking$ satisfies the axioms in definition \ref{defitransporter}.

(A1) This follows from axiom (A) on $\linking$.

(A2) By axiom (A) on $\linking$, we know that, for all $P, P' \in \linking$, $E(P) = Z(P)$ acts freely on $Mor_{\linking}(P,P')$ and that $\rho_{P,P'}$ is the orbit map of this action. Thus, we have to check that $E(P') = Z(P')$ acts freely on $Mor_{\linking}(P,P')$. Suppose $\varphi \in Mor_{\linking}(P,P')$ and $x \in E(P')$ are such that $\varepsilon_{P'}(x) \circ \varphi = \varphi$. Then, $x$ centralizes $\rho(\varphi)(P)$, so $x = \rho(\varphi)(y)$ for some $y \in Z(P)$, since $P$ is $\fusion$-centric. Hence, $\varphi = \delta_{P'}(x) \circ \varphi = \varphi \circ \delta_P(y)$ by axiom (C) for linking systems, and thus by axiom (A) we deduce that $y=1$, $x=1$ and the action is free.

(B) By construction of the functor $\varepsilon$, we know that $\varepsilon_{P,P'}:N_S(P,P') \longrightarrow Mor_{\linking}(P,P')$ is injective for all $P,P' \in \linking$. Thus, we have to check that the composite $\rho_{P,P'} \circ \varepsilon_{P,P'}$ sends $g \in N_S(P,P')$ to $c_g \in Hom_S(P,P')$. Note that the following holds for any $P,P' \in \linking$ and any $x \in N_S(P,P')$:
$$
\iota_{P'} \circ \varepsilon_{P,P'}(x) = \varepsilon_{P',S}(1) \circ \varepsilon_{P,P'}(x) = \varepsilon_{S}(x) \circ \varepsilon_{P,S}(1) = \delta_S(x) \circ \iota_P
$$
and hence so does the following on $\fusion$:
$$
incl^S_{P'} \circ \rho_{P,P'}(\varepsilon_{P,P'}(x)) = \rho_{P,S}(\iota_{P'}\circ \varepsilon_{P,P'}(x)) = \rho_{P,S}(\delta_S(x) \circ \iota_P) = c_x.
$$

(C) This follows from axiom (C) for linking systems.

(I) The group $Aut_{\linking}(S)$ has Sylow subgroups by Lemma 8.1 \cite{BLO3}, since $\delta(S)$ is normal in it and has finite index prime to $p$. This also proves  that $\delta(S)$ is a Sylow $p$-subgroup.

(II) Let $\varphi \in Iso_{\linking}(P,P')$, $P \lhd R$, $P' \lhd R'$ be such that $\varphi \circ \varepsilon_{P,P}(R) \circ \varphi^{-1} \leq \varepsilon_{P',P'}(R')$. We want to see that there exists $\widetilde{\varphi} \in Mor_{\linking}(R,R')$ such that $\widetilde{\varphi} \circ \varepsilon_{P,R}(1) = \varepsilon_{P',R'}(1) \circ \varphi$. Since $P'$ is $\fusion$-centric, it is fully $\fusion$-centralized. Then, we may apply axiom (II) for fusion systems to the morphism $f = \rho(\varphi)$, that is, $f$ extends to some $\widetilde{f} \in Hom_{\fusion}(N_f,S)$, where
$$
N_f = \{ g \in N_S(P) | f c_g f^{-1} \in Aut_S(P')\},
$$
and clearly $R \leq N_f$. Hence, $\widetilde{f}$ restricts to a morphism in $Hom_{\fusion}(R,S)$. Furthermore, $\widetilde{f}(R) \leq R'$ since $f$ conjugates $Aut_R(P)$ into $Aut_{R'}(P')$.

Now, $(\iota_{P',R'} \circ \varphi) \in Mor_{\linking}(P,R')$ is a lifting in $\linking$ for $incl^{R'}_{P'} \circ f \in Hom_{\fusion}(P,R')$, and we can fix a lifting $\psi \in Mor_{\linking}(R,R')$ for $\widetilde{f}$. Thus, by Lemma \ref{3.2OV} (i) there exists a unique $\widetilde{\iota} \in Mor_{\linking}(P,R)$, a lifting of $incl^R_P$, such that $\iota_{P',R'} \circ \varphi = \psi \circ \widetilde{\iota}$. Since $\rho(\widetilde{\iota}) = incl^R_P = \rho(\iota_{P,R})$, by axiom (A) it follows that there exists some $z \in Z(P)$ such that $\widetilde{\iota} = \iota_{P,R} \circ \delta_P(z) = \delta_R(\rho(\iota_{P,R})(z))\circ \iota_{P,R}$, where the second equality holds by axiom (C). Hence $\iota_{P',R'} \circ \varphi = (\psi \circ \delta_R(\rho(\iota_{P,R})(z))) \circ \iota_{P,R}$.

(III) Let $P_1 \leq P_2 \leq \ldots$ be an increasing sequence of objects in $\linking$, $P = \cup P_n$, and $\varphi_n \in Mor_{\linking}(P_n,S)$ satisfying $\varphi_n = \varphi_{n+1} \circ \iota_{P_n,P_{n+1}}$ for all $n$. We want to see that there exists some $\varphi \in Mor_{\linking}(P,S)$ such that $\varphi_n = \varphi \circ \iota_{P_n,P}$ for all $n$.

Set $f_n = \rho(\varphi_n)$ for all $n$. Then, by hypothesis, $f_n = f_{n+1} \circ incl^{P_{n+1}}_{P_n}$ for all $n$. Now, it is clear that $\{f_n\}$ forms a nonempty inverse system, and there exists $f \in Hom_{\fusion}(P,S)$ such that $f_n = f_{|P_n}$ for all $n$ (the existence follows from Proposition 1.1.4 in \cite{RZ}, and the fact that $f$ is a morphism in $\fusion$ follows from axiom (III) for fusion systems).

Consider now the following commutative diagram (in $\fusion$):
$$
\xymatrix{
 & P \ar[d]^{f} \\
P_1 \ar[ru]^{incl} \ar[r]_{f_1} & S\\
}
$$
The same arguments used to prove that axiom (II) for transporter systems holds on $\linking$ above apply now to show that there exists a unique $\varphi \in Mor_{\linking}(P,S)$ such that $\varphi_1 = \varphi \circ \iota_{P_1,P}$. Combining this equality with $\varphi_1 = \varphi_2 \circ \iota_{P_1,P_2}$ and Lemma \ref{3.2OV} (iv) (morphisms in $\linking$ are epimorphisms in the categorical sense), it follows that $\varphi_2 = \varphi \circ \iota_{P_2,P}$. Proceeding by induction it now follows that $\varphi$ satisfies the desired condition.

\end{proof}

Finally we state Proposition 3.6 from \cite{OV}, deeply related to Theorem A in \cite{BCGLO1}. These two results are only suspected to hold in the compact case, but yet no proof has been published. Before stating the result, we introduce some notation. For a finite group $G$, the subgroup $O_p(G) \leq G$ is the maximal normal $p$-subgroup of $G$.

\begin{defi}\label{hgenhsat}

Let $\fusion$ be a fusion system over a finite $p$-group $S$, and let $\hh \subseteq Ob(\fusion)$ be a subset of objects. Then, we say that $\fusion$ is \textbf{$\hh$-generated} if every morphism in $\fusion$ is a composite of restrictions of morphisms in $\fusion$ between subgroups in $\hh$, and we say that $\fusion$ is \textbf{$\hh$-saturated} if the saturation axioms hold for all subgroups in the set $\hh$.

\end{defi}

\begin{prop}

(3.6 \cite{OV}). Let $\fusion$ be a fusion system over a finite $p$-group $S$ (not necessarily saturated), and let $\transp$ be a transporter system associated to $\fusion$. Then, $\fusion$ is $Ob(\transp)$-saturated. If $\fusion$ is also $Ob(\transp)$-generated, and if $Ob(\transp) \supseteq Ob(\fusion^c)$, then $\fusion$ is saturated. More generally, $\fusion$ is saturated if it is $Ob(\transp)$-generated, and every $\fusion$-centric subgroup $P \leq S$ not in $Ob(\transp)$ is $\fusion$-conjugate to some $P'$ such that
$$
Out_S(P') \cap O_p(Out_{\fusion}(P')) \neq \{1\}.
$$

\end{prop}


\subsection{Unstable Adams operations on $p$-local compact groups} 

To conclude this section, we introduce unstable Adams operations for $p$-local compact groups and their main properties. Basically, we summarize the work from \cite{JLL} in order to give the proper definition of such operations and the main properties that we will use in later sections.

Let $(S, \fusion)$ be a saturated fusion system over a discrete $p$-toral group, and let $\theta: S \to S$ be a fusion preserving automorphism (that is, for each $f \in Mor(\fusion)$, the composition $\theta \circ f \circ \theta^{-1} \in Mor(\fusion)$). The automorphism $\theta$ naturally induces a functor on $\fusion$, which we denote by $\theta_{\ast}$, by setting $\theta_{\ast}(P) = \theta(P)$ on objects and $\theta_{\ast}(f) = \theta \circ f \circ \theta^{-1}$ on morphisms.

\begin{defi}\label{defiuao}

Let $\g = \ploc$ be a $p$-local compact group and let $\zeta$ be a $p$-adic unit. An \textbf{unstable Adams operation} on $\g$ of degree $\zeta$ is a pair $(\psi, \Psi)$, where $\psi$ is a fusion preserving automorphism of $S$, $\Psi$ is an automorphism of $\linking$, and such the following is satisfied:

\begin{enumerate}[(i)]

\item $\psi$ restricts to the $\zeta$ power map on $T$ and induces the identity on $S/T$;

\item for any $P \in Ob(\linking)$, $\Psi(P) = \psi(P)$;

\item $\rho \circ \Psi = \psi_{\ast} \circ \rho$, where $\rho: \linking \to \fusion$ is the projection functor; and

\item for each $P, Q \in Ob(\linking)$ and all $g \in N_S(P,Q)$, $\Psi(\delta_{P,Q}(g)) = \delta_{\psi(P), \psi(Q)}(\psi(g))$.

\end{enumerate}

In particular, $\Psi$ is an isotypical automorphism of $\linking$ in the sense of \cite{BLO3}.

\end{defi}

For a $p$-local compact group $\g$, let $\operatorname{Ad}(\g)$ be the group of unstable Adams operations on $\g$, with group operation the composition and the indentity functor as its unit. Also, for a positive integer $m$, let $\Gamma_m(p) \leq (\padic)^{\times}$ denote the subgroup of all $p$-adic units $\zeta$ of the form $1 + p^m \padic$.

Next, we state the existence of unstable Adams operations for all $p$-local compact groups. The following result corresponds to the second part of Theorem 4.1 \cite{JLL}.

\begin{thm}\label{existencepsi}

Let $\g$ be a $p$-local compact group. Then, for any sufficiently large positive integer $m$ there exists a group homomorphism
\begin{equation}\label{alpham}
\alpha: \Gamma_m(p) \longrightarrow \operatorname{Ad}(\g)
\end{equation}
such that, for each $\zeta \in \Gamma_m(p)$, $\alpha(\zeta) = (\psi, \Psi)$ has degree $\zeta$.

\end{thm}

There is an important property of unstable Adams operations which we will use repeatedly in the forthcoming sections. This was stated as Corollary 4.2 in \cite{JLL}.

\begin{prop}\label{finitesetinv}

Let $\g$ be a $p$-local compact group, and let $\pp \subseteq Ob(\linking)$ and $\mm \subseteq Mor(\linking)$ be finite subsets. Then, for any sufficiently large positive integer $m$, and for each $\zeta \in \Gamma_m(p)$, the group homomorphism $\alpha$ from (\ref{alpham}) satisfies $\alpha(\zeta)(P) = P$ and $\alpha(\zeta)(\varphi) = \varphi$ for all $P \in \pp$ and all $\varphi \in \mm$.

\end{prop}

\begin{rmk}

Let $(\psi, \Psi)$ be an unstable Adams operation on a $p$-local compact group $\g$. By point (iv) in \ref{defiuao}, $\Psi \circ \delta_S = \delta \circ \psi: S \to Aut_{\linking}(S)$, and hence the automorphism $\psi$ is completely determined by $\Psi$. Thus, for the rest of this paper we will make no mention of $\psi$ (unless necessary) and refer to the unstable Adams operation $(\psi, \Psi)$ just by $\Psi$.

\end{rmk}


\section{Families of operations and invariance}

Let $\g$ be a $p$-local compact group, and let $\Psi$ be an unstable Adams operation on $\g$. The degree of $\Psi$ will not be relevant in any of the constructions introduced in this section, and thus we will make no reference to it.

Let $S^{\Psi} \leq S$ be the subgroup of fixed elements of $S$ under the fusion preserving automorphism $\psi:S \to S$, that is,
$$
S^{\Psi} = \{x \in S \mbox{ } | \mbox{ } \psi(x) = x\},
$$
and more generally, for a subgroup $P \leq S$, let $P^{\Psi} = P \cap S^{\Psi}$.

\begin{rmk}

The action of $\Psi$ on the fusion system $\fusion$ is somehow too crude to allow us to see any structure on the fixed points, since for each $H \leq S^{\Psi}$,
$$
Aut_{\fusion}(H)^{\Psi}  \stackrel{def} = \{f \in Aut_{\fusion}(H) \mbox{ } | \mbox{ } \psi_{\ast}(f) = f\} = Aut_{\fusion}(H).
$$
We look then for fixed points in $\linking$.

\end{rmk}

\begin{lmm}\label{restrictmorph}

Let $\varphi: P \to Q$ be a $\Psi$-invariant morphism in $\linking$. Then, $\rho(\varphi)$ restricts to a morphism $f: P^{\Psi} \to Q^{\Psi}$ in the fusion system $\fusion$.

\end{lmm}

\begin{proof}

This follows by axiom (C) for linking systems, applied to each $\delta(x) \in \delta(P^{\Psi}) \leq Aut_{\linking}(P)$, since then both $\delta(x)$ and $\varphi$ are $\Psi$-invariant morphisms in $\linking$.

\end{proof}

The above Lemma justifies then defining the \textit{fixed points subcategory} of $\linking$ as the subcategory $\linking^{\Psi}$ with object set $Ob(\linking^{\Psi}) = \{P \in Ob(\linking) \mbox{ } | \mbox{ } \Psi(P) = P\}$ and with morphism sets
$$
Mor_{\linking^{\Psi}}(P, Q) = \{\varphi \in Mor_{\linking}(P, Q) \mbox{ } | \mbox{ } \Psi(\varphi) = \varphi\}.
$$
We can also define the \textit{fixed points subcategory} of $\fusion$ as the subcategory $\fusion^{\Psi}$ with object set the set of subgroups of $S^{\Psi}$, and such that
$$
Mor(\fusion^{\Psi}) = \gen{\{\rho(\varphi) \mbox{ } | \mbox{ } \varphi \in Mor(\linking^{\Psi})\}}.
$$
$\fusion^{\Psi}$ is, by definition, a fusion system over the finite $p$-group $S^{\Psi}$, but the category $\linking^{\Psi}$ is far from being a transporter system associated to it.

\begin{rmk}\label{trouble}

This way of considering fixed points has many disadvantages. For instance, there is no control on the morphism sets $Hom_{\fusion^{\Psi}}(P,Q)$, since given a subgroup $H \leq S^{\Psi}$ there might be several subgroups $P \in Ob(\linking^{\Psi})$ such that $P \cap S^{\Psi} = H$. It becomes then rather difficult to check any of the saturation axioms on $\fusion^{\Psi}$. Another issue is the absence of an obvious candidate of a transporter system associated to $\fusion^{\Psi}$.

\end{rmk}

To avoid the problems listed in Remark \ref{trouble}, we can try different strategies. For instance, instead of considering a single operation $\Psi$ acting on $\g$, we can consider a (suitable) family of operations $\{\Psi_i\}_{i \in \N}$ on $\g$. This will be specially useful when proving that certain properties hold after suitably increasing the power of $\Psi$. The situation is improved when we restrict our attention to the full subcategory $\linking^{\bullet} \subseteq \linking$, since, by \cite{JLL}, unstable Adams operations are completely determined by its action on $\linking^{\bullet}$. We can also restrict the morphism sets that we consider as fixed by imposing stronger invariance conditions.


\subsection{A family of operations}

Starting from the unstable Adams operation $\Psi$, we consider an specific family of operations which will satisfy our purposes. Set first $\Psi_0 = \Psi$, and let $\Psi_{i+1} = (\Psi_i)^p$, that is, the operation $\Psi_i$ iterated $p$ times. Consider the resulting family $\{\Psi_i\}_{i \in \N}$ fixed for the rest of this section, and note that, if an object or a morphism in $\linking$ is fixed by $\Psi_i$ for some $i$, then it is fixed by $\Psi_j$ for all $j \geq i$.

\begin{rmk}\label{BI1}

By Corollary 4.2 \cite{JLL}, we may assume that there exist a subset $\hh \subseteq Ob(\linking^{\bullet})$ of representatives of the $S$-conjugacy classes in $\linking^{\bullet}$ and a set $\widehat{\mm} = \bigcup_{P,R \in \hh} \hatmm_{P,R}$, where each $\hatmm_{P,R} \subseteq Iso_{\linking}(P,R)$ is a set of representatives of the elements in $Rep_{\fusion}(P,R)$, and such that

\begin{enumerate}[(i)]

\item $\Psi(P) = P$ for all $P \in \pp$; and

\item $\Psi(\varphi) = \varphi$ for all $\varphi \in \hatmm$.

\end{enumerate}

\end{rmk}

Let us also fix some notation. For each $i$, set
$$
S_i \stackrel{def} = \{x \in S \mbox{ } | \mbox{ } \Psi_i(x) = x\},
$$
and more generally, for each subgroup $R \leq S$, set $R_i = R \cap S_i$. In particular, the notation $T_i$ means the subgroup of $T$ (the maximal torus) of fixed elements under $\Psi_i$ rather than the subgroup of $T$ of exponent $p^i$. There will not be place for confusion about such notation in this paper.

\begin{lmm}

For each $i$, $T_i \lneqq T_{i+1}$, and hence $T = \bigcup_{i \in \N} T_i$.

\end{lmm}

As a consequence, we deduce the following.

\begin{prop}

The following holds in $\linking$.

\begin{enumerate}[(i)]

\item Let $P \in Ob(\linking)$. Then, there exists some $M_P$ such that, for all $i \geq M_P$, $P$ is $\Psi_i$-invariant.

\item Let $\varphi \in Mor(\linking)$. Then, there exists some $M_{\varphi}$ such that, for all $i \geq M_{\varphi}$, $\varphi$ is $\Psi_i$-invariant.

\end{enumerate}

\end{prop}

Next we provide a tool to detect $\Psi_i$-invariant morphisms in $\linking^{\bullet}$. Note that, for any morphism $\varphi \in Mor(\linking)$, the following holds by Proposition 3.3 \cite{BLO3}.
$$
\Psi_i(\varphi) = \varphi \Longrightarrow \Psi_i(\varphi^{\bullet}) = \varphi^{\bullet}.
$$
Our statement is proved by comparing morphisms in $\linking^{\bullet}$ to the representatives fixed in $\hatmm$, which we know to be $\Psi_i$-invariant (for all $i$) \textit{a priori}.

\begin{lmm}\label{detectmorph}

Let $P, R$ be representatives fixed in \ref{BI1}, and let $Q \in \conj{P}{S}$ and $Q' \in \conj{R}{S}$. Then, a morphism $\varphi' \in Iso_{\linking}(Q,Q')$ is $\Psi_i$-invariant if and only if for all $a \in N_S(P,Q)$ there exist $b \in N_S(R,Q')$ and a morphism $\varphi \in \hatmm_{P,R}$ such that
\begin{enumerate}[(i)]

\item $\varphi' = \delta(b) \circ \varphi \circ \delta(a^{-1})$; and

\item $\delta(b^{-1} \cdot \Psi_i(b)) \circ \varphi = \varphi \circ \delta(a^{-1} \cdot \Psi_i(a))$.

\end{enumerate}

\end{lmm}

\begin{proof}

Note that condition (ii) above is equivalent to
\begin{enumerate}[(ii')]

\item $\delta(\Psi_i(b) \cdot b^{-1}) \circ \varphi' = \varphi' \circ \delta(\Psi_i(a) \cdot a^{-1})$.

\end{enumerate}
Suppose first that $\varphi'$ is $\Psi_i$-invariant. Choose $x \in N_S(P, Q)$ and $y \in N_S(R,Q')$, and set $\phi = \delta(y^{-1}) \circ \varphi' \circ \delta(x)$. Then, there exist $\varphi \in \hatmm_{P,R}$ such that $[\rho(\varphi)] = [\rho(\phi)] \in Rep_{\fusion}(P,R)$ and $z \in R$ such that $\varphi = \delta(z) \circ \phi$.

Let then $a = x \in N_S(P,Q)$ and $b = y \cdot z^{-1} \in N_S(R, Q')$. This way, condition (i) is satisfied, and we have to check that condition (ii) is also satisfied. Since both $\varphi$ and $\varphi'$ are $\Psi_i$-invariant, we may apply $\Psi_i$ to (i) to get the following equality
$$
\delta(b) \circ \varphi \circ \delta(a^{-1}) = \varphi' = \Psi_i(\varphi') = \delta(\Psi_i(b)) \circ \varphi \circ \delta(\Psi_i(a)^{-1}),
$$
which is clearly equivalent to condition (ii) since morphisms in $\linking$ are epimorphisms in the categorical sense.

Suppose now that condition (i) and (ii) are satisfied for certain $a$, $b$ and $\varphi$. Write $\varphi = \delta(b^{-1}) \circ \varphi' \circ \delta(a)$ and apply $\Psi_i$ to this equality. Since $\varphi$ is $\Psi_i$-invariant, we get
$$
\delta(\Psi_i(b)^{-1}) \circ \Psi_i(\varphi') \circ \delta(\Psi_i(a)) = \delta(b^{-1}) \circ \varphi' \circ \delta(a).
$$
Thus, after reordering the terms in this equation and using condition (ii') above, we obtain
$$
\Psi_i(\varphi') \circ \delta(\Psi_i(a) \cdot a^{-1}) = \varphi' \circ \delta(\Psi_i(a) \cdot a^{-1}),
$$
which implies that $\Psi_i(\varphi') = \varphi'$ since morphisms in $\linking$ are epimorphisms in the categorical sense.

\end{proof}


\subsection{A stronger invariance condition}

Given an arbitrary $\Psi_i$-invariant object $P$ in $\linking^{\bullet}$, there is no way \textit{a priori} of relating $P_i$ to $P$, not to say of comparing $C_S(P_i)$ or $N_S(P_i)$ to $C_S(P)$ or $N_S(P)$ respectively. This turns out to be crucial if we want to study fixed points on $\g$ under the operation $\Psi_i$. This is the reason why we now introduce a stronger invariance condition for an object in $\linking^{\bullet}$ to be $\Psi_i$-invariant. This is a condition on all objects in $\fusion^{\bullet}$.

\begin{defi}\label{Kdeterm}

Let $K \leq S$ be a subgroup. We say that a subgroup $P \in Ob(\fusion^{\bullet})$ is \textbf{$K$-determined} if
$$
(P \cap K)^{\bullet} = P.
$$
For a $K$-determined subgroup $P \leq S$ we call the subgroup $P \cap K$ the \textbf{$K$-root} of $P$.

\end{defi}

Our interest lies on the case $K = S_i$, in which case $\Psi_i$-invariance is a consequence.

\begin{lmm}

Let $P \in Ob(\fusion^{\bullet})$ be an $S_i$-determined subgroup for some $i$. Then, $P$ is $\Psi_i$-invariant.

\end{lmm}

\begin{proof}

Note that, if $(P_i)^{\bullet} = P$, then $P = P_i \cdot T_P$, where $T_P$ is the maximal torus of $P$, since $P_i$ is a finite subgroup of $P$. Thus, by applying $\Psi_i$ to $P$, we get
$$
\Psi_i(P) = \Psi_i(P_i \cdot T_P) = P_i \cdot T_P = P,
$$
since $\Psi_i(x) = x$ for all $x \in P_i$ and $\Psi_i(T_P) = T_P$ by definition of $\Psi_i$.

\end{proof}

We prove now that actually $S_i$-determined subgroups exist (for $i$ big enough).

\begin{lmm}

Let $P \in Ob(\fusion^{\bullet})$. Then, there exists some $M_P \geq 0$ such that, for all $i \geq M_P$, $P$ is $S_i$-determined.

\end{lmm}

\begin{proof}

Let $T_P$ be the maximal torus of $P$, and note that $P = \cup P_i$. Thus, there exists some $M$ such that, for all $i \geq M$, $R_i$ contains representatives of all the elements of the finite group $R/T_R$.

Since $P^{\bullet} = P$ and $Aut_{\fusion}(T)$ is a finite group, it follows then that there must exist some $M_P \geq M$ such that, for all $i \geq M_P$, $(P_i)^{\bullet} = P$.

\end{proof}

We can then assume that all the objects fixed in Remark \ref{BI1} are $S_i$-determined for all $i$, since there are only finitely many of them in the set $\hh$.

One must be careful at this point. Given $P, R$ $S_i$-determined subgroups, if $P_i$ and $R_i$ are $\fusion$-conjugate, then the properties of $\functor$ imply that so are $P$ and $R$, but the converse is not so straightforward.

\begin{lmm}\label{rootconj}

There exists some $M_1 \geq 0$ such that, for all $P \in \hh$ and all $i \geq M_1$, if $Q \in \conj{P}{S}$ is $S_i$-determined then $Q_i$ is $S$-conjugate to $P_i$.

\end{lmm}

\begin{proof}

Since $\hh$ contains finitely many $S$-conjugacy classes of subgroups and each $S$-conjugacy class contains finitely many $T$-conjugacy classes of subgroups, it is enough to prove the statement for a single $T$-conjugacy class, say $\conj{P}{T}$.

Given such subgroup $P$, let $\pi = P/(P \cap T) \leq S/T$, and let $\widetilde{P} \leq S$ be the pull-back of $S \rightarrow S/T \leftarrow \pi$. Then, for any $Q \in \conj{P}{T}$, the following clearly holds: $Q \cap T = P \cap T$, $Q/(Q\cap T) = P/(P\cap T)$ and $Q \leq \widetilde{P}$.

For any section $\sigma: \pi \to \widetilde{P}$ of the projection $\widetilde{P} \to \pi$, let $Q_{\sigma} = (P \cap T) \cdot \gen{\sigma(\pi)} \leq \widetilde{P}$. Given a random section $\sigma$, the subgroup $Q_{\sigma}$ will not in general be in the $T$-conjugacy class of $P$, but it is clear that for every $Q \in \conj{P}{T}$ there exists some $\sigma$ such that $Q = Q_{\sigma}$.

Now, up to $T$-conjugacy, the set of sections $\sigma: \pi \to \widetilde{P}$ is in one to one correspondence with the cohomology group $H^1(\pi; T)$, which is easily proved to be finite by an standard transfer argument. Thus, we can fix representatives $\sigma_1, \ldots, \sigma_l$ of $T$-conjugacy classes of sections such that $Q_{\sigma_j} \in \conj{P}{T}$ for all $j$. For each such section, let $H_j = \gen{\sigma_j(\pi)} \leq \widetilde{P}$. It is clear then that there exists some $M_P$ such that, for all $i \geq M_P$, $H_1, \ldots, H_l \leq S_i$ and $Q_{\sigma_1}, \ldots, Q_{\sigma_l}$ are all $S_i$-determined.

Let now $Q \in \conj{P}{T}$ be $S_i$-determined. In particular, this means that there exists a section $\sigma: \pi \to \widetilde{P} \cap S_i$ such that $Q = Q_{\sigma}$. Such a section is $T$-conjugate to some $\sigma_j$ in the list of representatives previously fixed, namely there exists some $t \in T$ such that $\sigma = c_t \circ \sigma_j$. Note that this implies that $H_{\sigma} \stackrel{def} = \gen{\sigma(\pi)} \in \conj{H_j}{T}$, and hence $Q \in \conj{Q_{\sigma_j}}{T}$. To finish the proof, note that $Q_i = (P \cap T_i) \cdot H_{\sigma}$ and $(Q_{\sigma_j})_i = (P \cap T_i) \cdot H_j$, and clearly $t \in T$ conjugates $(Q_{\sigma_j})_i$ to $Q_i$. 

\end{proof}

We now prove some properties of $S_i$-determined subgroups.

\begin{prop}\label{propcentral}

There exists some $M_2 \geq 0$ such that, for all $i \geq M_2$, if $P$ is $S_i$-determined, then
$$
C_S(P_i) = C_S(P).
$$

\end{prop}

\begin{proof}

Let $\mathfrak{X}$ be a set of representatives of the $S$-conjugacy classes in $Ob(\fusion^{\bullet})$, and note that this is a finite set by Lemma 3.2 (a) in \cite{BLO3}.

For any $P \in \mathfrak{X}$, consider the set $\{T_R \mbox{ } | \mbox{ } R \in \conj{P}{S}\}$ of maximal tori of subgroups in $\conj{P}{S}$. This is a finite set, since, for any two $R, Q \in \conj{P}{S}$ and any $f \in Iso_{\fusion}(R,Q)$ the isomorphism $f_{|T_R}: T_R \to T_Q$ has to be the restriction of an automorphism of $Aut_{\fusion}(T)$, by Lemma 2.4 (b) \cite{BLO3}, and $Aut_{\fusion}(T)$ is a finite group. It is clear then that there exists some $M_P$ such that, for all $i \geq M_P$ and all $R \in \conj{P}{S}$,
$$
C_S((T_R)_i) = C_S(T_R).
$$

Let now $i \geq M_P$ and let $R \in \conj{P}{S}$ be $S_i$-determined. We can then write $R = R_i \cdot T_R$ and $R_i = R_i \cdot (T_R)_i$, and it follows that
$$
C_S(R) = C_S(R_i) \cap C_S(T_R) = C_S(R_i) \cap C_S((T_R)_i) = C_S(R_i).
$$
The proof is finished by taking $M_2 = \operatorname{max} \{M_P \mbox{ } | \mbox{ } P \in \mathfrak{X}\}$.

\end{proof}

The following result is an easy calculation which is left to the reader.

\begin{lmm}\label{rootconjx}

Let $P,Q \leq S$ be $S_i$-determined subgroups such that $Q_i \in \conj{P_i}{S}$. Then, for all $x \in N_S(P_i,Q_i)$,
$$
x^{-1} \cdot \Psi_i(x) \in C_T(P).
$$

\end{lmm}

Since, for any $H \leq K \leq S$ we have $C_K(H) = K \cap C_S(H)$, the following are immediate consequences of Proposition \ref{propcentral}.

\begin{cor}\label{corcentral1}

Let $i \geq M_2$ and let $P$ be $S_i$-determined. If $C_S(P) = Z(P)$, then $C_{S_i}(P_i) = Z(P_i)$.

\end{cor}

\begin{cor}\label{corcentral2}

There exists some $M_3 \geq 0$ such that, for all $i \geq M_3$, if $Q$ is $S_i$-determined and $C_S(Q) \gneqq Z(Q)$, then $C_{S_i}(Q_i) \gneqq Z(Q_i)$.

\end{cor}

\begin{proof}

As usual, since $Ob(\fusion^{\bullet})$ contains finitely many $T$-conjugacy classes of subgroups $P$ such that $C_S(P) \gneqq Z(P)$, it is enough to prove the statement for a single $T$-conjugacy class of such subgroups.

Fix such a subgroup $P$. We can assume that the statement holds for $P$, and let $z \in C_S(P) \setminus Z(P)$ be such that $z \in C_{S_i}(P_i) \setminus Z(P_i)$ (such an element exists by Proposition \ref{propcentral}). Let now $Q \in \conj{P}{T}$ be $S_i$-determined, and let $x \in N_S(P_i, Q_i)$ (such an element exists by Lemma \ref{rootconj}). Let also $z' = xzx^{-1} \in C_S(Q) \setminus Z(Q)$. If $C_T(P) \leq (Z(P) \cap T)$, then, by Lemma \ref{rootconj},
$$
z \cdot (x^{-1} \Psi_i(x)) = (x^{-1} \Psi_i(x)) \cdot z,
$$
which implies that $\Psi_i(z') = z' \in C_{S_i}(Q_i) \setminus Z(Q_i)$ by Lemma \ref{detectmorph}. In this case, let $M_P = M_2$ as in Proposition \ref{propcentral}.

On the other hand, if $Z(P) \cap T \lneqq C_T(P)$, then we can take the element $z$ above to be in $C_T(P) \setminus (Z(P) \cap T)$. It is clear then that there exists some $M_P$ such that, for all $i \geq M_P$, $z \in C_{T_i}(P) \setminus (Z(P) \cap T_i)$, in which case $z' = z \in C_{T_i}(Q) \setminus (Z(Q) \cap T_i)$. The proof is finished then by taking $M_3$ to be the maximum of the $M_P$ among a finite set of representatives.

\end{proof}

We can also relate the normalizer of $P_i$ to the normalizer of $P$.

\begin{prop}\label{propnormal}

There exists some $M_4 \geq 0$ such that, for all $i \geq M_4$, if $P$ is $S_i$-determined, then
$$
N_S(P_i) \leq N_S(P).
$$

\end{prop}

This is not obvious at all, since the properties of $\functor$ only tell us that, for $x \in N_S(P_i)$, there exists a unique $f \in Aut_{\fusion}(P)$ extending the isomorphism $c_x \in Aut_{\fusion}(P_i)$.

\begin{proof}

Again, it is enough to check the statement for a single $S$-conjugacy class of objects in $\fusion^{\bullet}$. Let then $P$ be a representative of such an $S$-conjugacy class, and consider the set $\{T_R \mbox{ } | \mbox{ } R \in \conj{P}{S}\}$ (which is a finite set, as we have shown in the proof for Proposition \ref{propcentral}).

It follows then that there exists some $M_4$ such that, for all $i \geq M_4$ and all $R \in \conj{P}{S}$, if $g \in N_S((T_R)_i)$ then $g \in N_S(T_R)$. The proof is finished since, for $R \in \conj{P}{S}$ which is $S_i$-determined ($i \geq M_4$), there is an equality $R = R_i \cdot T_R$.

\end{proof}


\section{Strongly fixed points}

Using the notion of $S_i$-determined subgroups we introduce the \textit{strongly fixed points} of $\g$ under the action of $\Psi_i$, and prove their main properties. In particular this section contains the proof of Theorem \ref{ThA}.

For each $i$, consider the sets
\begin{equation}\label{generatorsets}
\begin{array}{c}
\hh_i^{\bullet} \stackrel{def} = \{R \in Ob(\linking^{\bullet}) \mbox{ } | \mbox{ } R \mbox{ is } S_i\mbox{-determined}\},\\
\hh_i \stackrel{def} = \{R_i = R \cap S_i \mbox{ } | \mbox{ } R \in \hh_i^{\bullet}\},\\
\end{array}
\end{equation}
and note that the functor $\functor$ gives a one-to-one correspondence between these two sets. Let also $\widehat{\hh}_i$ be the closure of $\hh_i$ by overgroups in $S_i$. Finally, for each pair $P, R \in \hh_i^{\bullet}$, consider the sets
$$
\begin{array}{c}
Mor_{\linking, i}(P,R) = \{\varphi \in Mor_{\linking}(P, R) \mbox{ } | \mbox{ } \Psi_i(\varphi) = \varphi\},\\
Hom_{\fusion, i}(P,R) = \{\rho(\varphi) \mbox{ } | \mbox{ } \varphi \in Mor_{\linking,i}(P,R)\}.
\end{array}
$$

Recall from Lemma \ref{restrictmorph} that, given a $\Psi_i$-invariant morphism $\varphi:P \to R$ in $\linking$, the homomorphism $f = \rho(\varphi)$ restricts to a homomorphism $f_i: P_i \to R_i$. Thus, we can consider, for each pair $P_i, R_i \in \hh_i$, the set
$$
A(P_i, R_i) = \{ f_i = res^P_{P_i}(f) \mbox{ } | \mbox{ } f \in Hom_{\fusion, i}(P,R) \} \subseteq Hom_{\fusion}(P_i, R_i).
$$
Again, the functor $\functor$ provides a bijection from the above set to $Hom_{\fusion, i}(P,R)$.

\begin{defi}

For each $i$, the \textbf{$i$-th strongly fixed points fusion system} is the fusion system $\fusion_i$ over $S_i$ whose morphisms are compositions of restrictions of morphisms in $\{A(P_i, R_i) \mbox{ } | \mbox{ } P_i,R_i \in \hh_i\}$.

\end{defi}

The category $\fusion_i$ is indeed a fusion system over $S_i$, as well as a fusion subsystem of $\fusion$.

Let $\linking_i^{\circ}$ be the category with object set $\hh_i$ and whose morphism sets are spanned by the sets $Mor_{\linking,i}(P,R)$, after identifying the sets $\hh_i$ and $\hh_i^{\bullet}$ via $\functor$. The category $\linking_i^{\circ}$ is well-defined since, in fact, it can be thought as a subcategory of $\linking$, although its actual definition will make more sense for the purposes of this paper.

We want now to close $\linking_i^{\circ}$ by overgroups, and one has to be careful at this step. Let $H, K \in \widehat{\hh}_i$ be arbitrary subgroups, and let $P, R \in \hh_i^{\bullet}$ be such that $H \leq P$, $K \leq R$. We say then that a morphism  $\varphi \in Mor_{\linking,i}(P,R)$ \textit{restricts to a morphism} $\varphi: H \to K$ if $f = \rho(\varphi):P \to R$ restricts to a homomorphism $f_{|H}: H \to K$ in $\fusion$. We need a technical lemma before we define the closure of $\linking_i^{\circ}$ by overgroups.

\begin{lmm}\label{bulletSideterm}

For any subgroup $H \leq S_i$, the subgroup $H^{\bullet} \leq S$ is $S_i$-determined. If, in addition, $H \in \widehat{\hh_i}$, then $H^{\bullet}$ is $\fusion$-centric.

\end{lmm}

\begin{proof}

To show that $H^{\bullet}$ is $S_i$-determined, we have to prove that $(H^{\bullet} \cap S_i)^{\bullet} = H^{\bullet}$. Since $H \leq H^{\bullet} \cap S_i \leq H^{\bullet}$, the equality follows by applying $\functor$ to these inequalities. The centricity of $H^{\bullet}$ when $H \in \widehat{\hh_i}$ follows by definition of the set $\widehat{\hh}_i$ and by Proposition 2.7 \cite{BLO3}. 

\end{proof}

As a consequence of this result, for any $H \in \widehat{\hh}_i$, the subgroup $(H^{\bullet} \cap S_i) \in \hh_i$.

\begin{defi}\label{defigi}

For each $i$, the \textbf{$i$-th strongly fixed points transporter system} is the category $\linking_i$ with object set $\widehat{\hh}_i$  and with morphism sets
$$
Mor_{\linking_i}(H, K) = \{\varphi \in Mor_{\linking,i}(H^{\bullet}, K^{\bullet}) \mbox{ } | \mbox{ } \varphi \mbox{ restricts to a morphism } \varphi:H \to K\}.
$$

Finally, the \textbf{$i$-th strongly fixed points system} is the triple $\g_i = (S_i, \fusion_i, \linking_i)$.

\end{defi}

The composition rule in $\linking_i$ is induced by the composition rule in $\linking$, and hence is well-defined. $\linking_i$ is called a transporter system since we will prove in this section that it actually has such structure. 


\subsection{Properties of the strongly fixed points subsystems}

We now study the properties of each of the triples $\g_i$ defined above. At some point this will require increasing the degree of the initial operation $\Psi$ again, and also fix some more objects and morphisms in $\linking$, apart from those already fixed in Remark \ref{BI1}. First, we describe some basic properties of the triples $\g_i$, most of which are inherited from the properties of $\linking$.

\begin{lmm}\label{morphi}

For all $i$ and for all $P_i, R_i \in \hh_i$, there are equalities
\begin{enumerate}[(i)]

\item $A(P_i, R_i) = Hom_{\fusion_i}(P_i, R_i)$; and

\item $Mor_{\linking,i}(P,R) = Mor_{\linking_i}(P_i, R_i)$.

\end{enumerate}

\end{lmm}

\begin{proof}
by definition of $\fusion_i$ and $\linking_i$, it is enough to show only point (ii). The proof is done then by induction on the order of the subgroups $P_i, R_i \in \hh_i$.

First, we consider the case $P_i = R_i = S_i$. This case is obvious since in $\hh_i$ the subgroup $S_i$ has no overgroups. Consider now a pair $P_i, R_i \lneqq S_i$. There is an obvious inclusion $Mor_{\linking,i}(P,R) \subseteq Mor_{\linking_i}(P_i, R_i)$ by definition of $\linking_i$. On the other hand, any morphism in $Mor_{\linking_i}(P_i, R_i)$ is a composition of restrictions of $\Psi_i$-invariant morphisms in $\linking$, by the induction hypothesis, and hence we have the equality.

\end{proof}

The category $\linking_i$ also has associated a functor $\functor_i$, induced by the original $\functor$ in $\linking$. Next we describe this functor in $\linking_i$ and its main properties, most of which are identical to those of $\functor$. Define first $\functor_i$ on an object $H_i \in Ob(\linking_i)$ as
$$
(H_i)^{\bullet}_i \stackrel{def} = (H_i)^{\bullet} \cap S_i.
$$
Recall that by definition, $Mor_{\linking_i}(H_i, K_i) = \{ \omega \in Mor_{\linking,i}(H_i^{\bullet}, K_i^{\bullet}) \mbox{ } | \mbox{ } \omega \mbox{ restricts to } \omega: H_i \to K_i\}$. Thus, on a morphism $\varphi \in Mor_{\linking_i}(H_i, K_i)$, $\functor_i$ is defined as the unique $\varphi \in Mor_{\linking_i}(H_i^{\bullet} \cap S_i, K_i^{\bullet} \cap S_i)$ which restricts to $\varphi: H_i \to K_i$. Note that in particular $\functor_i$ is the identity on $\hh_i$ by construction.

\begin{prop}\label{propertiesbulleti}

The following holds for $\functor_i$.
\begin{enumerate}[(i)]

\item For all $H_i \leq S_i$, $((H_i)^{\bullet}_i)^{\bullet}_i = (H_i)^{\bullet}_i$.

\item If $H_i \leq K_i \leq S_i$, then $(H_i)^{\bullet}_i \leq (K_i)^{\bullet}_i$.

\item Every morphism $\varphi \in Mor_{\linking_i}(H_i,K_i)$ extends to a unique $(\varphi)^{\bullet}_i \in Mor_{\linking_i}((H_i)^{\bullet}_i, (K_i)^{\bullet}_i)$

\end{enumerate}

In particular, $\functor_i$ is a functor from $\linking_i$ to $\linking_i^{\circ}$ which is left adjoint to the inclusion $\linking_i^{\circ} \subseteq \linking_i$.

\end{prop}

\begin{proof}

(i) Let $H_i \leq S_i$ be any subgroup. By Lemma \ref{bulletSideterm}, $H \stackrel{def} = (H_i)^{\bullet}$ is $S_i$-determined, and hence
$$
\xymatrix@C=4mm{
(H_i)^{\bullet}_i \stackrel{def} = H \cap S_i \ar[rr]^{\functor} & & (H \cap S_i)^{\bullet} = H \ar[rr]^{\underline{\phantom{A}} \cap S_i} & & H \cap S_i \stackrel{def} = ((H_i)^{\bullet}_i)^{\bullet}_i.\\
}
$$

(ii) This property follows from Lemma 3.2 (c) \cite{BLO3}.

(iii) Both the existence and uniqueness of $(\varphi)^{\bullet}_i$ hold by definition of $\linking_i$.

It follows now that $\functor_i$ is a functor, and by Lemma \ref{bulletSideterm}, it sends objects and morphisms in $\linking_i$ to objects and morphisms in $\linking_i^{\circ}$. The adjointness property holds since the restriction map
$$
Mor_{\linking_i}((H_i)^{\bullet}_i, P_i) \stackrel{res} \longrightarrow Mor_{\linking_i}(H_i, P_i)
$$
is a bijection for all $H_i \in \widehat{\hh_i}$ and all $P_i \in \hh_i$.

\end{proof}

\begin{cor}\label{inclequiv}

The inclusion $\linking_i^{\circ} \subseteq \linking_i$ induces a homotopy equivalence
$$
|\linking_i^{\circ}| \simeq |\linking_i|.
$$

\end{cor}

\begin{proof}

It is a consequence of Corollary 1 \cite{Quillen}, since the inclusion $\linking_i^{\circ} \subseteq \linking_i$ has a right adjoint by Proposition \ref{propertiesbulleti}.

\end{proof}

\begin{thm}\label{transpi}

There exists some $M_{\Psi} \geq 0$ such that, for all $i \geq M_{\Psi}$, $\linking_i$ is a transporter system associated to $\fusion_i$.

\end{thm}

\begin{proof}

Clearly, each $\linking_i$ is a nonempty finite category with $Ob(\linking_i) \subseteq Ob(\fusion_i)$. The first step to prove the statement is to define the pair of functors $\varepsilon_i: \transp_{Ob(\linking_i)}(S_i) \to \linking_i$ and $\rho_i: \linking_i \to \fusion_i$, but actually these two functors are naturally induced by $\linking$. Indeed, the functor $\varepsilon: \transp_{Ob(\linking)}(S) \to \linking$ restricts to a functor $\varepsilon_i$ as above. With respect to the functor $\rho_i$, the projection functor $\rho: \linking \to \fusion$ naturally induces a functor
$$
\rho_i: \linking_i \to \fusion_i
$$
by the rule $\rho_i(\varphi: H \to H') = res^P_H(\rho(\varphi))$ for each morphism $\varphi \in Mor_{\linking_i}(H, H')$ (Lemma \ref{restrictmorph}).

We next proceed to prove that all the axioms for transporter systems in definition \ref{defitransporter} are satisfied (after considering a suitable power of $\Psi$). Note first that using the functor $\functor_i$, it is enough to prove the axioms on the subset $\hh_i \subseteq Ob(\linking_i)$. Recall by Proposition \ref{3.5OV} that $\linking$ satisfies the axioms of a transporter system.

Axioms (A1), (A2), (B) and (C) hold by definition of $\g_i$ and by the same axioms on $\linking$. Next we show that axiom (I) holds. By Proposition \ref{morphi}, there is an equality $Aut_{\linking_i}(S_i) = Mor_{\linking,i}(S,S) \leq Aut_{\linking}(S)$. Furthermore, by definition of all these groups, and because we have fixed representatives of the elements of $Out_{\fusion}(S)$ in $\hatmm$, there is a group extension
$$
1 \to (\varepsilon_i)_{S_i,S_i}(S_i) \to Aut_{\linking_i}(S_i) \to Out_{\fusion}(S) \to 1.
$$
Thus, since $\{1\} \in Syl_p(Out_{\fusion}(S))$, the axiom follows. There is no need of checking that axiom (III) of transporter systems holds in this case, since $\linking_i$ is a finite category.

Axiom (II) will be proved by steps, since we need to discard finitely many of the first operations in $\{\Psi_i\}$. We recall here its statement. 

\begin{itemize}

\item[(II)] Let $\varphi \in Iso_{\linking_i}(P_i,Q_i)$, $P_i \lhd \widetilde{P}_i \leq S_i$ and $Q_i \lhd \widetilde{Q}_i \leq S_i$ be such that $\varphi \circ \varepsilon_i(\widetilde{P}_i) \circ \varphi^{-1} \leq \varepsilon_i(\widetilde{Q}_i)$. Then, there is some $\widetilde{\varphi} \in Mor_{\linking_i}(\widetilde{P}_i, \widetilde{Q}_i)$ such that $\widetilde{\varphi} \circ \varepsilon_i(1) = \varepsilon_i(1) \circ \varphi$.

\end{itemize}

\noindent The proof of axiom (II) is then organized as follows. First, we fix a finite list of representatives of all possible such extensions in $\linking$ (up to conjugacy by an element in $S$). In the second step we prove the axiom for the representatives fixed in the set $\hh$ of Remark \ref{BI1}, and finally in the third step we prove the general case.

\begin{itemize}

\item \textbf{Step 1}. Representatives of the extensions.

\end{itemize}

Let $P_i, Q_i \in \hh_i$, $P_i \lhd \widetilde{P}_i$, $Q_i \lhd \widetilde{Q}_i$ and $\varphi \in Mor_{\linking_i}(Pi, Q_i)$ as in the statement of axiom (II). By definition of $\linking_i$ and $\hh_i$, it is equivalent to consider in $\linking$ the corresponding situation: $\varphi: P \to Q$ such that $\varphi \circ \varepsilon(\widetilde{P}) \circ \varphi^{-1} \leq \varepsilon(\widetilde{Q})$, where $P = (P_i)^{\bullet}$, $Q = (Q_i)^{\bullet}$, $\widetilde{P} = (\widetilde{P}_i)^{\bullet}$ and $\widetilde{Q} = (\widetilde{Q}_i)^{\bullet}$. Note that, by Lemma \ref{bulletSideterm}, the subgroups $P, Q, \widetilde{P}$ and $\widetilde{Q}$ are $S_i$-determined. We have then translated a situation in $\linking_i$ to a situation in $\linking$. We will keep this notation for the rest of the proof.

Note first that if $P \in Ob(\linking)$, then the quotient $N_S(P)/P$ is finite. Indeed, since $C_S(P) = Z(P)$, it follows that $N_S(P)/P \cong (N_S(P)/Z(P)) / (P/Z(P)) = Out_S(P) \leq Out_{\fusion}(P)$, and this group is finite by axiom (I) for saturated fusion systems (or by Proposition 2.3 \cite{BLO3}).

As a consequence, if we fix $P, Q \in Ob(\linking)$ and a morphism $\varphi \in Mor_{\linking}(P,Q)$, then, up to conjugacy by elements in $S$, there are only finitely many morphisms $\widetilde{\varphi}: \widetilde{P} \to \widetilde{Q}$ such that $P \lhd \widetilde{P}$, $Q \lhd \widetilde{Q}$ and $\widetilde{\varphi}$ extends $\varphi$.

Consider then the sets $\hh$ and $\hatmm$ fixed in Remark \ref{BI1}. For each morphism $\varphi: P \to Q$ fixed in $\hatmm$, we can fix representatives (up to $S$-conjugacy) of all possible extensions $\widetilde{\varphi}: \widetilde{P} \to \widetilde{Q}$. Let $\widetilde{\mm}$ be s set of all such representatives:
$$
\widetilde{\mm} \stackrel{def} = \{ \widetilde{\varphi}: \widetilde{P} \to \widetilde{Q} \mbox{ } | \mbox{ } \varphi \in \hatmm\}.
$$

It is clear then that there exists some $M_{\Psi}$ such that, for all $i \geq M_{\Psi}$, the following holds:
\begin{enumerate}[(i)]

\item each extension $\widetilde{\varphi}$ in the above set is $\Psi_i$-invariant;

\item for each such $\widetilde{\varphi}$, the source subgroup, $\widetilde{P}$, is $S_i$-determined; and

\item each $\widetilde{P}$ is $S_i$-conjugate to the corresponding representative of $\conj{\widetilde{P}}{S}$ fixed in $\hh$.

\end{enumerate}

\begin{itemize}

\item \textbf{Step 2}. Both $P$ and $Q$ are in the set $\hh$ fixed in \ref{BI1}.

\end{itemize}

In this case, there are $\varphi' \in \hatmm$ and $x \in Q$ such that $\varphi = \delta(x) \circ \varphi'$. Furthermore, since both $\varphi$ and $\varphi'$ are $\Psi_i$-invariant, so is $\delta(x)$, and hence $x \in Q_i$. Let then $\widetilde{\varphi}' \in \widetilde{\mm}$ be the extension of $\varphi'$ which sends $\widetilde{P}$ to $x \cdot \widetilde{Q} \cdot x^{-1}$, and let
$$
\widetilde{\varphi} = \delta(x) \circ \widetilde{\varphi}'.
$$
It follows then that $\widetilde{\varphi}: \widetilde{P} \to \widetilde{Q}$ is an extension of $\varphi$, which in addition is $\Psi_i$-invariant since both $\delta(x)$ and $\widetilde{\varphi}'$ are. We just have to consider the corresponding morphism in $\linking_i$ to prove that axiom (II) holds in this case, since $\widetilde{P}$ and $\widetilde{Q}$ are $S_i$-determined.

\begin{itemize}

\item \textbf{Step 3}. One (or possibly both) of the subgroups $P, Q$ is not in $\hh$.

\end{itemize}

Since $\varphi$ is $\Psi_i$-invariant, it follows from Lemma \ref{detectmorph} that there exist subgroups $R, R' \in \hh$, a morphism $\varphi' \in \hatmm_{R,R'}$, and elements $a \in N_S(R,P)$ and $b \in N_S(R',Q)$ such that
\begin{enumerate}[(i)]

\item $\varphi = \delta(b) \circ \varphi' \circ \delta(a^{-1})$, and

\item $\delta(b^{-1} \cdot \Psi_i(b)) \circ \varphi ' = \varphi' \circ \delta(a^{-1} \cdot \Psi_i(a))$.

\end{enumerate}

Let then $\widetilde{R} = a^{-1} \cdot \widetilde{P} \cdot a$ and $\widetilde{R}' = b^{-1} \cdot \widetilde{Q} \cdot b$, and let $\widetilde{\varphi}': \widetilde{R} \to \widetilde{R}'$ be the extension of $\varphi'$ fixed in $\widetilde{\mm}$. Let also
$$
\widetilde{\varphi} \stackrel{def} = \delta(b) \circ \widetilde{\varphi}' \circ \delta(a^{-1}): \widetilde{P} \to \widetilde{Q}.
$$
Since $\varphi$ is $\Psi_i$-invariant, and by Lemma 4.3 \cite{BLO3}, it follows then that $\widetilde{\varphi}$ is also $\Psi_i$-invariant. Since $\widetilde{P}, \widetilde{Q}$ are $S_i$-determined, the proof is finished by considering then the morphism induced in $\linking_i$ by $\widetilde{\varphi}$.

\end{proof}

By Corollary \ref{inclequiv}, the following statement is still true after replacing $Ob(\linking_i)$ by the subset $\hh_i$.

\begin{cor}

For each $i \geq M_{\Psi}$, $\fusion_i$ is $Ob(\linking_i)$-generated and $Ob(\linking_i)$-saturated.

\end{cor}

\begin{proof}

The $Ob(\linking_i)$-generation of $\fusion_i$ follows by definition of $\g_i$, and the $Ob(\linking_i)$-saturation of $\fusion_i$ follows directly by Proposition 3.6 \cite{OV}.

\end{proof}

Next we describe an interesting property of the family $\{\g_i\}$. Recall that, by Corollary \ref{inclequiv}, each inclusion $\linking_i^{\circ} \subseteq \linking_i$ induces a homotopy equivalence $|\linking_i^{\circ}| \simeq |\linking_i|$, and the striking point is that for each $i$ there is also a faithful functor $\Theta_i: \linking_i^{\circ} \to \linking_{i+1}^{\circ}$ defined by:
\begin{equation}\label{thetai}
\xymatrix@R=1mm{
\linking_i^{\circ} \ar[rr]^{\Theta_i} & & \linking_{i+1}^{\circ}\\
P_i \ar@{|->}[rr] & & (P_i)^{\bullet} \cap S_{i+1} = P_{i+1}\\
\varphi \ar@{|->}[rr] & & \varphi.\\
}
\end{equation}
It is easy to check that this is a well-defined functor: since $P_i \in \hh_i$ is the (unique) $S_i$-root of the $S_i$-determined subgroup $P \in Ob(\linking^{\bullet})$, it follows by construction that $P_{i+1}$ is an element of the set $\hh_{i+1}$, and by definition of $\linking_i$ and $\linking_{i+1}$, there is a natural inclusion of sets
$$
Mor_{\linking_i}(P_i,R_i) \subseteq Mor_{\linking_{i+1}}(P_{i+1}, Q_{i+1})
$$
for all $P_i, Q_i \in \hh_i$, which is a group monomorphism whenever $P_i = Q_i$. It follows then that $\Theta_i$ is faithful for all $i$.

Note that in general the functor $\Theta_i$ does not induce a commutative square
$$
\xymatrix@R=1mm@C=1mm{
\linking_i \ar[dd]_{\rho_i} \ar[rr]^{\Theta_i} & & \linking_{i+1} \ar[dd]^{\rho_{i+1}}\\
 & \times & \\
\fusion_i \ar[rr]_{incl_i} & & \fusion_{i+1}.\\
}
$$
For instance, whenever $S$ has positive rank, we have $\rho_{i+1}(\Theta_i(S_i)) = (S_i)^{\bullet} \cap S_{i+1} = S_{i+1} \gneqq S_i = incl(\rho(S_i))$. This is not a great inconvenience, as we prove below. Let $\fusion_i^{\hh_i} \subseteq \fusion_i$ be the full subcategory with object set $\hh_i$, and let $\theta_i: \fusion_i^{\hh_i} \to \fusion_{i+1}$ be the functor induced by $\Theta_i$.

\begin{prop}

For all $i$, there is a natural transformation $\tau_i$ between the functors $incl_i$ and $\theta_i$.

\end{prop}

\begin{proof}

Set $\tau_i(P_i) = [incl_i(P_i) = P_i \hookrightarrow P_{i+1} = (P_i)^{\bullet} \cap S_i]$ for each $P_i \in \hh_i$, and set also
$$
\xymatrix@R=1mm@C=1mm{
P_i \ar[rr]^{incl_i} \ar[dd]_{f_i} & & P_{i+1} \ar[dd]^{f_{i+1}} \\
 & \tau_i(f_i) & \\
R_i \ar[rr]_{incl_i} & & R_{i+1}\\
}
$$
for each $(f_i: P_i \to R_i) \in Mor(\fusion_i^{\hh_i})$, where $f_{i+1} = res^{(P_i)^{\bullet}}_{P_{i+1}}((f_i)^{\bullet})$. This is well-defined since $S_i$-roots are unique, and because of the properties of $\functor$. In particular, it follows from Proposition 3.3 \cite{BLO3} that the above square is always commutative, and hence $\tau_i$ is a natural transformation.

\end{proof}

This way, we have a sequence of maps
$$
\xymatrix@R=5mm@C=1cm{
\ldots \ar[r] & |\linking_{i-1}^{\circ}| \ar[r]^{|\Theta_{i-1}|} \ar[d]_{\simeq} & |\linking_i^{\circ}| \ar[r]^{|\Theta_i|} \ar[d]_{\simeq} & |\linking_{i+1}^{\circ}| \ar[r]^{|\Theta_{i+1}|} \ar[d]_{\simeq} & \ldots, \\
 & |\linking_{i-1}| & |\linking_i| & |\linking_{i+1}| & \\
}
$$
and we can ask about the homotopy colimit of this sequence. Let $I$ be the poset of the natural numbers with inclusion, and let $\Theta: I \to \operatorname{Top}$ be the functor defined by $\Theta(i) = |\linking^{\circ}_i|$ and $\Theta(i \to i+1) = |\Theta_i|$.

\begin{thm}\label{hocolim1}

There is a homotopy equivalence
$$
(\operatorname{hocolim}_{\rightarrow I} \Theta)^{\wedge}_p \simeq B\g.
$$

\end{thm}

\begin{proof}

The statement follows since, as categories, $\linking^{\bullet} = \bigcup_{i \in \N} \linking_i^{\circ}$.

\end{proof}

We finally study the elements of a set $\hh_i$ as objects in $\fusion_j$ for $j \geq i$.

\begin{prop}\label{centricquasicentric}

Let $P_i \in \hh_i$. Then, the following holds:
\begin{enumerate}[(i)]

\item $P_i$ is $\fusion_i$-centric,

\item $P_i$ is $\fusion_j$-quasicentric for all $j \geq i$, and

\item $P_i$ is $\fusion$-quasicentric.

\end{enumerate}

\end{prop}

\begin{proof}

Property (i) is a consequence of Proposition \ref{propcentral}. Properties (ii) and (iii) are consequence of the Proposition below.

\end{proof}

\begin{prop}

Let $P \leq S$ be $\fusion$-quasicentric and $S_i$-determined for some $i$. Then,
\begin{enumerate}[(i)]

\item $P_i$ is $\fusion$-quasicentric, and

\item $P_i$ is $\fusion_j$-quasicentric for all $j \geq i$.

\end{enumerate}

\end{prop}

\begin{proof}

The proof is done by steps.

\begin{itemize}

\item \textbf{Step 1}. Let $H \in \conj{P_i}{\fusion}$ and $Q = (H)^{\bullet}$. Then there are equalities
$$
C_S(H) = C_S(Q).
$$

\end{itemize}

\noindent Indeed, for $P_i$ the equality holds by Proposition \ref{propcentral}, since $P$ is $S_i$-determined. Also, since $P$ is $S_i$-determined, we can write
$$
\begin{array}{ccc}
P = P_i \cdot T_P & \mbox{and} & P_i = P_i \cdot (T_P)_i.
\end{array}
$$
Let now $H \in \conj{P_i}{\fusion}$, $Q = (H)^{\bullet}$, and let $f \in Iso_{\fusion}(P_i,H)$. Using the infinitely $p$-divisibility property of $T_P$, we can write then $Q = f(P_i) \cdot T_Q = H \cdot T_Q$ and $H = f(P_i) \cdot f((T_P)_i) = H \cdot (T_Q)_i$. Thus,
$$
C_S(H) = C_S(H) \cap C_S((T_Q)_i) = C_S(H) \cap C_S(T_Q) = C_S(Q),
$$
where the second equality holds by (the proof of) Proposition \ref{propcentral}, since we had previously fixed representatives of all the $S$-conjugacy classes in $\conj{P}{\fusion}$ in \ref{BI1}.

\begin{itemize}

\item \textbf{Step 2}. For each $H \in \conj{P_i}{\fusion}$, $H$ is $\fusion$-quasicentric.

\end{itemize}

Let $C_{\fusion}(H)$ be the centralizer fusion system of $H$, and note that, in particular, $C_{\fusion}(H)$ is a fusion system over $C_S(H) = C_S(Q)$. Let also $f: R \to R'$ be a morphism in $C_{\fusion}(H)$. By definition of $C_{\fusion}(H)$, there is a morphism $\widetilde{f}: R \cdot H \to R' \cdot H$ in $C_{\fusion}(H)$ which extends $f$ and such that it restricts to the identity on $H$.

By applying $\functor$ to $\widetilde{f}$, we obtain a new morphim $(\widetilde{f})^{\bullet}$ which restricts to $f^{\bullet}: (R)^{\bullet} \to (R')^{\bullet}$ and to the identity on $(H)^{\bullet} = Q$. It follows then that $f^{\bullet}$ is a morphism in $C_{\fusion}(Q)$. On the other hand, there is an obvious inclusion of categories $C_{\fusion}(Q) \subseteq C_{\fusion}(H)$, which is in fact an equality by the above. Since $Q$ is $\fusion$-quasicentric by hypothesis, the proof of Step 2 is finished.

\begin{itemize}

\item \textbf{Step 3}. For each $j \geq i$ and each $H \in \conj{P_i}{\fusion_j}$, $H$ is $\fusion_j$-quasicentric.

\end{itemize}

This case follows by Step 1, together with the properties of the functor $\functor$, since we can identify $C_{\fusion_j}(H)$ with a subcategory of $C_{\fusion}(H)$.

\end{proof}


\subsection{Consequences of the existence of approximations by $p$-local finite groups} 

We have skipped in the previous section the issue of the saturation of the fusion systems $\fusion_i$. This is a rather difficult question and we want to discuss it apart from the main results. In this section we will also study some consequences of the case when the triples $\g_i$ are $p$-local finite groups. Examples of this situation will be described in the following section.

Recall that we have used Proposition 3.6 \cite{OV} to prove that for each $i$ the fusion system $\fusion_i$ is $Ob(\linking_i)$-generated and $Ob(\linking_i)$-saturated. Recall also that Proposition 3.6 \cite{OV} gives conditions for the fusion systems $\fusion_i$ to be saturated: each $\fusion_i$-centric subgroup $H \leq S_i$ not in $Ob(\linking_i)$ has to be $\fusion_i$-conjugate to some $K \leq S_i$ such that
\begin{equation}\label{condition}
Out_{S_i}(K) \cap O_p(Out_{\fusion_i}(K)) \neq \{1\}.
\end{equation}

The disadvantage of proving the saturation of $\fusion_i$ by means of this result lies obviously on the difficulty in checking the above condition, but the advantage of proving saturation using it is also great, since in particular this would mean that all $\fusion_i$-centric $\fusion_i$-radical subgroups are in $Ob(\linking_i)$. Indeed, note that if there was some $\fusion_i$-centric $\fusion_i$-radical not in $Ob(\linking_i)$, then the category $\linking_i^{\circ}$ defined in the previous section could not be extended to a whole centric linking system associated to $\fusion_i$ (at least in an obvious way), and the functors $\Theta_i$ would not be valid any more.

In order to check the condition above, we can consider the following two situations:
\begin{enumerate}[(a)]

\item $H$ is not an $S_i$-root, that is, $H \lneqq (H)^{\bullet} \cap S_i$; or

\item $H$ is an $S_i$-root, that is, $H = (H)^{\bullet} \cap S_i$ but $(H)^{\bullet}$ is not $\fusion$-centric.

\end{enumerate}
The difficult case to study is (b), but we can prove rather easily that condition (\ref{condition}) is always satisfied in case (a).

\begin{prop}\label{nosiroot}

Let $H \leq S_i$ be an $\fusion_i$-centric subgroup not in $Ob(\linking_i)$ and such that $H \lneqq (H)^{\bullet} \cap S_i$. Then, $H$ satisfies condition (\ref{condition}).

\end{prop}

\begin{proof}

Let $K \stackrel{def} = (H)^{\bullet} \cap S_i \leq S_i$. The functor $\functor_i$ provides a natural inclusion $Aut_{\fusion_i}(H) \leq Aut_{\fusion_i}(K)$. Consider also the subgroup $A = \{c_x \in Aut_{\fusion_i}(H) \mbox{ } | \mbox{ } x \in N_K(H)\}$. Via the above inclusion of automorphism groups in $\fusion_i$, we can see $A$ as
$$
A = Aut_{\fusion_i}(H) \cap Inn(K).
$$
Since, by hypothesis, $H \lneqq K$, it follows that $H \lneqq N_K(H)$, and hence $Inn(H) \lneqq A$, since $H$ is $\fusion_i$-centric by hypothesis.

The group $Aut_{\fusion_i}(H)$, seen as a subgroup of $Aut_{\fusion_i}(K)$, normalizes $Inn(K)$, and thus $A \lhd Aut_{\fusion_i}(H)$ and
$$
\{1\} \neq A/Inn(H) \leq O_p(Out_{\fusion_i}(H)).
$$
Also, by definition of $A$, there is an inclusion $A/Inn(H) \leq Out_{S_i}(H)$ and this finishes the proof.

\end{proof}

It is still an open question whether condition (\ref{condition}) is satisfied in case (b) in general.

\begin{defi}

Let $\g$ be a $p$-local compact group, and let $\Psi$ be an unstable Adams operation acting on $\g$. We say that $\Psi$ \textbf{approximates $\g$ by $p$-local finite groups} if there exists some $M_{\Psi}'$ such that, for all $i \geq M_{\Psi}'$, condition (\ref{condition}) holds for all $H \in Ob(\fusion_i^c) \setminus Ob(\linking_i)$. We also say then that $\Psi$ \textbf{induces an approximation of $\g$ by $p$-local finite groups}.

\end{defi}

\begin{cor}

Let $\g$ be a $p$-local compact group such that, for all $P \in Ob(\fusion^{\bullet}) \setminus Ob(\linking^{\bullet})$, $C_S(P) \gneqq Z(P)$. Then, any unstable Adams operation $\Psi$ induces an approximation of $\g$ by $p$-local finite groups.

\end{cor}


\subsection{The Stable Elements Theorem}

When $\g$ is approximated by $p$-local finite groups, we can prove the Stable Elements Theorem (5.8 \cite{BLO2}) for $\g$. Such result holds, for instance, for the examples in the forthcoming section of this paper.

\begin{prop}\label{inandout}

Let $\g$ be a $p$-local compact group, and let $\Psi$ be an unstable Adams operation that approximates $\g$ by $p$-local finite groups. Then, there are natural isomorphisms
$$
\begin{array}{ccc}
H^{\ast}(BS; \F_p) \cong \varprojlim H^{\ast}(BS_i; \F_p) & \mbox{ and } & H^{\ast}(B\g; \F_p) \cong \varprojlim H^{\ast}(B\g_i; \F_p).
\end{array}
$$

\end{prop}

\begin{proof}

Let $X$ be either $B\g$ or $BS$, and similarly let $X_i$ be either $B\g_i$ or $BS_i$, depending on which case we want to prove. Consider also the homotopy colimit spectral sequence for cohomology (XII.5.7 \cite{BK}):
$$
E^{r,s}_2 = \varprojlim \!\! \phantom{i}^rH^s(X_i;\F_p) \Longrightarrow H^{r+s}(X;\F_p).
$$
We will see that, for $r \geq 1$, $E_2^{r,s} = \{0\}$, which, in particular, will imply the statement.

For each $s$, let $H^s_i = H^s(X_i;\F_p)$, and let $F_i$ be the induced morphism in cohomology (in degree $s$) by the map $|\Theta_i|$. The cohomology ring $H^{\ast}(X_i;\F_p)$ is noetherian by Theorem 5.8 \cite{BLO2}, and in particular $H^s_i$ is a finite $\F_p$-vector space for all $s$ and all $i$. Thus, the inverse system $\{H^s_i;F_i\}$ satisfies the Mittag-Leffler condition (3.5.6 \cite{Weibel}), and hence the higher limits $\varprojlim^rH^s_i$ vanish for all $r \geq 1$. This in turn implies that the differentials in the above spectral sequence are all trivial, and thus it collapses.

\end{proof}

\begin{thm}\label{stable}

(Stable Elements Theorem for $p$-local compact groups). Let $\g$ be a $p$-local compact group, and suppose that there exists $\Psi$, an unstable Adams operation on $\g$, that approximates $\g$ by $p$-local finite groups. Then, the natural map
$$
H^{\ast}(B\g; \F_p) \stackrel{\cong} \longrightarrow H^{\ast}(\fusion) \stackrel{def} = \varprojlim_{\mathcal{O}(\fusion^c)} H^{\ast}(\underline{\phantom{A}}; \F_p) \subseteq H^{\ast}(BS; \F_p)
$$
is an isomorphism.

\end{thm}

\begin{proof}

Since each $\g_i$ is a $p$-local finite group (for $i$ big enough), we can apply the Stable Elements Theorem for $p$-local finite groups, Theorem 5.8 \cite{BLO2}: there is a natural isomorphism
$$
H^{\ast}(B\g_i; \F_p) \stackrel{\cong} \longrightarrow H^{\ast}(\fusion_i) = \varprojlim_{\mathcal{O}(\fusion_i^c)} H^{\ast}(\underline{\phantom{A}}; \F_p) \subseteq H^{\ast}(BS_i; \F_p)
$$
Thus, by Proposition \ref{inandout}, there are natural isomorphisms
$$
H^{\ast}(B\g; \F_p) \cong \varprojlim H^{\ast}(B\g_i; \F_p) \cong \varprojlim H^{\ast}(\fusion_i) \subseteq \varprojlim H^{\ast}(BS_i; \F_p) \cong H^{\ast}(BS;\F_p).
$$
Furthermore, the functor $\functor$ induces inclusions $\mathcal{O}(\fusion_i^c) \subseteq \mathcal{O}(\fusion_{i+1}^c)$ in a similar fashion as it induced the functors $\Theta_i$, and $\mathcal{O}(\fusion^{\bullet c}) = \bigcup_{i \in \N} \mathcal{O}(\fusion_i^c)$, from where it follows that
$$
\varprojlim_i H^{\ast}(\fusion_i) \stackrel{def} = \varprojlim_i \varprojlim_{\mathcal{O}(\fusion_i^c)} H^{\ast}(\underline{\phantom{A}}; \F_p) \cong \varprojlim_{\mathcal{O}(\fusion^c)} H^{\ast}(\underline{\phantom{A}}; \F_p) \stackrel{def} = H^{\ast}(\fusion).
$$

\end{proof}

\begin{rmk}

A general proof (i.e. for all $p$-local compact groups) of the above result would lead to a proof of Theorem 6.3 \cite{BLO2} in the compact case, just by doing some minor modifications in the proof for the finite case. This in turn would allow us to reproduce (most of) the work in \cite{BCGLO2} for $p$-local compact groups.

\end{rmk}

\begin{rmk}

Suppose $\g$ is approximated by $p$-local finite groups. Then, by Proposition \ref{centricquasicentric}, together with Theorem B \cite{BCGLO1}, we can define a zig-zag
$$
\xymatrix@R=4mm@C=1cm{
\ldots \ar[rd] & |\linking_{i-1}| \ar[d]^{\simeq} \ar[rd] & |\linking_i| \ar[d]^{\simeq} \ar[rd] & |\linking_{i+1}| \ar[d]^{\simeq} \ar[rd] & \\
 & |\linking_{i-1}^q| & |\linking_i^q| & |\linking_{i+1}^q| & \ldots\\
}
$$
where, for each $i$, $\linking_i^q$ is the quasi-centric linking system associated to $\linking_i$ (see \cite{BCGLO1}). This yields another homotopy colimit, which is easily seen to be equivalent to that in Theorem \ref{hocolim1}.

\end{rmk}


\section{Examples of approximations by $p$-local finite groups} 

We discuss now some examples of $p$-local compact groups which are approximated by $p$-local finite groups. The first example we consider is that of $p$-local compact groups of rank $1$, which will require rather descriptive arguments. The second example is that of $p$-local compact groups induced by the compact Lie groups $U(n)$. In this case, the particular action of $S/T$ over $T$ will be the key.


\subsection{$p$-local compact groups of rank $1$}

The main goal of this section then is to prove the following.

\begin{thm}\label{rank1}

Let $\g$ be a $p$-local compact group of rank $1$. Then, every unstable Adams operation $\Psi$ approximates $\g$ by $p$-local finite groups.

\end{thm}

To prove this result we will first study some technical properties of rank $1$ $p$-local compact groups, and then apply these properties to show that condition (\ref{condition}) holds always. This process will imply again fixing some finite list of objects and morphisms in $\linking$ and increasing the degree of $\Psi$ so that some properties hold. The approach here is rather exhaustive, and is not appropriate to study a more general situation. All the results achieved in the previous section are assumed to hold already.

\begin{rmk}\label{noclassification}

Possibly the main difficulty in this section is the absence of any kind of classification of rank $1$ $p$-local compact groups which we could use to reduce to a finite list of cases to study. In this sense, an attempt of a classification was made in \S 3 \cite{Gonza}, but only with partial results which are of no use here. Namely, the author proved that every rank $1$ $p$-local compact group uniquely determines a \textit{connected} rank $1$ $p$-local compact group which is in fact derived from either $S^1$, $SO(3)$ or $S^3$ (the last two only occurring for $p=2$), but there is still no reasonable notion equivalent to the group of components in classical Lie group theory.

Note that the above list of connected $p$-local compact groups does not contain the Sullivan spheres. This is because of the notion of connectivity used, which was rather strict but lead to stronger results, such as Corollary 3.2.5 \cite{Gonza}, which cannot be extended to weaker notions of connectivity.

\end{rmk}

Roughly speaking, the condition for a morphism (in $\linking$) to be $\Psi_i$-invariant is related to the existence of morphisms (in $\fusion$) sending elements of $T$ to elements outside $T$ (Lemma \ref{detectmorph}). Define then $S_0 \leq S$ as \textit{the} minimal strongly $\fusion$-closed subgroup of $S$ containing $T$. It is clear by definition of $S_0$ that such subgroup (if exists) is unique.

\begin{lmm}

Let $\g$ be a $p$-local compact group. Then, $S_0$ always exists. Furthermore, each element $x \in S_0$ is $\fusion$-subconjugate to $T$.

\end{lmm}

\begin{proof}

To prove the existence of $S_0$, it is enough to consider the intersection of all the strongly $\fusion$-closed subgroups of $S$ containing $T$, since the intersection of two strongly $\fusion$-closed subgroups is again strongly $\fusion$-closed.

To prove the second part of the statement, let $S_0' = \cup P_n$, where $P_0 = T$, and $P_{n+1}$ is the subgroup of $S$ generated by $P_n$ together with all the elements of $S$ which are $\fusion$-subconjugate to $P_n$. This is clearly an strongly $\fusion$-closed subgroup, hence $S_0 \leq S_0'$. On the other hand, if $S_0 \lneqq S_0'$, then there exists $x \in S_0' \setminus S_0$ and a morphism $f: \gen{x} \to T$ in $\fusion$ contradicting the fact that $S_0$ is strongly $\fusion$-closed.

\end{proof}

Next we describe the possible isomorphism types of $S_0$ in the rank $1$ case. The following criterion will be useful.

\begin{lmm}

Let $\g$ be a $p$-local compact group, and let $P \leq S$ be $\fusion$-subconjugate to $T$. Then,
$$
C_P(T) \stackrel{def} = C_S(T) \cap P = T \cap P.
$$

\end{lmm}

This Lemma can be understood as follows. If $x \in S$ is $\fusion$-conjugate to an element in $T$, then either $x$ is already an element in $T$ or $x$ acts nontrivially on $T$.

\begin{proof}

Let $f:P \to T$ be a morphism in $\fusion$. We can assume without loss of generality that $P = \gen{x}$ and that $P' = f(P)$ is fully $\fusion$-centralized, since it is a subgroup of $T$. This way we can apply axiom (II) for saturated fusion systems to $f$ to see that it extends to a morphism $\widetilde{f} \in Hom_{\fusion}(C_S(P) \cdot P, S)$.

Suppose then that $x$ acts trivially on $T$. In particular, $T \leq C_S(P)$, and thus in particular $\widetilde{f}$ restricts to $\widetilde{f}: T \cdot P \to S$. The infinitely $p$-divisibility of $T$ and the hypothesis on $f$ imply then that $\widetilde{f}(T) = T$ and $\widetilde{f}(P) \leq T$ respectively, and hence $P \leq T$.

\end{proof}

The above result implies that the quotient $S_0/T$ can be identified with a subgroup of $Aut(T) = GL_r(\padic)$, where $r$ is the rank of $T$. When $r=1$,
$$
Aut(T) \cong \left\{
\begin{array}{ll}
\Z/2 \times \adic, & p=2,\\
\Z/(p-1) \times \padic, & p>2,\\
\end{array}
\right.
$$
and we can prove the following.

\begin{lmm}\label{isotype}

Let $\g$ be a rank $1$ $p$-local compact group.
\begin{enumerate}[(i)]

\item If $p >2$, then $S_0 = T$.

\item If $p=2$, then $S_0$ has the isomorphism type of either $T$, $\dihed = \cup D_{2^n}$ or $\quat = \cup Q_{2^n}$.

\end{enumerate}

\end{lmm}

\begin{proof}

The case $p>2$ is immediate, since $Aut(T)$ does not contain any finite $p$-subgroup, and hence $S_0/T$ has to be trivial. Suppose then the case $p=2$. In this case, $Aut(T)$ contains a finite $2$-subgroup isomorphic to $\Z/2$, and hence either $S_0/T = \{1\}$ or $S_0/T \cong \Z/2$.

If $S_0/T = \{1\}$, then $S_0 = T$ and there is nothing to prove. Suppose otherwise that $S_0/T \cong \Z/2$. Then $S_0$ fits in an extension $T \to S_0 \to \Z/2$. By IV.4.1 \cite{MacLane} and II.3.8 \cite{Adem-Milgram}, the group
$$
H^2(\Z/2;T^{\tau}) \cong \Z/2
$$
classifies all possible extensions $T \to S_0 \to S_0/T$ up to isomorphism. Here, the superindex on $T$ means that the coefficients are twisted by the action of $\Z/2$ on $T$. Thus, up to isomorphism, there are only two possible discrete $2$-toral groups of rank $1$ with the desired action on $T$ and such that $S_0/T \cong \Z/2$, and the proof is finished since both $\dihed$ and $\quat$ satisfy these conditions and are non-isomorphic.

\end{proof}

The proof of Theorem \ref{rank1} will be done by cases, depending on the isomorphism type of $S_0$.

As happened when proving that $\linking_i$ is a transporter system associated to $\fusion_i$ (Theorem \ref{transpi}), proving Theorem \ref{rank1} will require fixing some finite list of objects and morphisms in $\fusion$ and considering operations $\Psi_i$ of degree high enough.

\begin{rmk}\label{BI3}

More specifically, we fix
\begin{enumerate}[(i)]

\item a set $\pp'$ of representatives of the $S$-conjugacy classes of non-$\fusion$-centric objects in $\fusion^{\bullet}$; and

\item for each pair $H, K \in \pp'$ such that $K$ is fully $\fusion$-normalized, a set $\mm_{H,K} \subseteq Hom_{\fusion}(H,K)$ of the classes in $Rep_{\fusion}(H,K)$.

\item for each $f \in \mm_{H,K}$ above, an ``Alperin-like'' decomposition (Theorem \ref{Alperin})
\begin{equation}\label{Alperinlike}
\xymatrix@R=2mm@C=4mm{
 & & & & & & & & & & & \\
 & L_1 \ar[r]^{\gamma_1} & L_1 & & L_2 \ar[r]^{\gamma_2} & L_2 & & & & L_k \ar[r]^{\gamma_k} & L_k & \\
 & & & & & & & & & & & \\
R_0 \ar[rrr]_{f_1} \ar[ruu] & & & R_1 \ar[luu] \ar[ruu] \ar[rrr]_{f_2} & & & R_2 \ar[luu] \ar[r] & \ldots \ar[r] & R_{k-1} \ar[ruu] \ar[rrr]_{f_k} & & & R_k, \ar[luu], \\
}
\end{equation}
where $R_0 = H$, $R_k = K$, $L_j$ is $\fusion$-centric $\fusion$-radical and fully $\fusion$-normalized for $j = 1, \ldots, k$, and
$$
f_R = f_k \circ f_{k-1} \circ \ldots \circ f_2 \circ f_1.
$$

\item for each $\gamma_j$ above, a lifting $\varphi_j$ in $\linking$.

\end{enumerate}
This is clearly a finite list, and hence by Proposition \ref{finitesetinv}, there exists some $M_{\Psi}' \geq 0$ such that, for all $i \geq M_{\Psi}'$, all the subgroups in $\pp'$ are $S_i$-determined and all the morphisms $\varphi_j$ are morphisms in $\linking_i$.

\end{rmk}

\begin{lmm}\label{morphTid}

Let $H \in \pp'$ be $S$-centric, and let $K \in \pp' \cap \conj{H}{\fusion}$ be any non-$S$-centric object. Then, the set $Hom_{\fusion}(H,K)$ contains an element $f$, together with a decomposition as (\ref{Alperinlike}),  such that for all $j=1, \ldots, k$
$$
f_j(C_T(R_{j-1})) \leq T.
$$

\end{lmm}

\begin{proof}

Suppose first that $S_0 = T$. Since in this case $T$ is strongly $\fusion$-closed, the condition holds by axiom (C) for linking systems. Also if $S_0 \cong \quat$ it is easy to see that $T_1 \leq T$ (the order $2$ subgroup of $T$) is strongly $\fusion$-closed (in fact it is $\fusion$-centra), and either $C_T(R) = T$ or $C_T(R) = T_1$. In both cases then the statement follows easily by axiom (C) of linking systems and the properties of $T$ and $T_1$.

We are thus left to consider the case $S_0 \cong \dihed$. Note that in this case every element in the quotient $S/S_0$ acts trivially on $T$. Also, $Z(S) \cap T = T_1$, but now this subgroup is not strongly $\fusion$-closed, and the subgroups $T_n$, $n \geq 2$, are all weakly $\fusion$-closed (this holds since the only elements of $S_0$ of order $2^n$ are all in $T$).

Set for simplicity $L = L_j$, $\varphi = \varphi_j$ and $f = f_j$. If $C_T(L) \geq T_n$ for some $n \geq 2$, then the condition above holds directly by axiom (C) for linking systems, since $T_n$ is weakly $\fusion$-closed. We can assume thus that $C_T(L) = T_1$. Even more, if $L \cap S_0 = T_1$, then the condition above still holds since $S_0$ is strongly $\fusion$-closed.

By inspection of $S_0$, this leaves only one case to deal with
$$
L_0 \stackrel{def} = L \cap S_0 = \gen{x, T_1} = R_{j-1} \cap S_0 = R_j \cap S_0 \cong \Z/2 \times \Z/2
$$
for some element $x$ which has order $2$. If we set $t_2$ for a generator of $T_2 \leq T$, then it is also easy to check that $t_2$ normalizes $L_0$, and in fact, since $L/L_0$ acts trivially on $T$, it also normalizes $L, R_{j-1}$ and $R_j$.

Set also $t_1$ for the generator of $T_1$. The automorphism group of $L_0$ is isomorphic to $\Sigma_3$, generated by $c_{t_2}$ together with an automorphism $f_0$ of order $3$ which sends $t_1$ to $x$ and $x$ to $xt_1$. Note that the assumption that $S_0 \cong \dihed$ implies that $Aut_{\fusion}(L_0) \cong \Sigma_3$. For the purposes of the proof we can now assume that  $f$ restricts to $f_0$.

Let then $\omega = f^{-1} \circ c_{t_2} \circ f^{-1} \circ c_{t_2}^{-1}$. It is easy to see that $\omega$ induces the identity on $L/L_0$, and by inspecting the automorphism group of $L_0$ it follows that $\omega_{|L_0} = f_0$.

Consider now $f' = \omega^{-1} \circ f$. By definition, $f'$ induces the same automorphism on $L/L_0$ as $f$, and the identity on $L_0$. To show that we can replace $f$ by $f'$ we have to show that the image of $R_{j-1}$ by $f$ and $f'$ are the same:
$$
\xymatrix{
R_{j-1} \ar[r]^{f} & R_j \ar[r]^{f} & R_j' \ar[r]^{c_{t_2}^{-1}} & R_j' \ar[r]^{f^{-1}} & R_j \ar[r]^{c_{t_2}} & R_j,\\
}
$$
where $R_j' = f(R_j)$ is normalized by $t_2$ by the above arguments.

\end{proof}

We can assume then that, for each pair $H,K \in \pp'$ (with $C_S(K)\geq Z(K)$)  the set $\mm_{H,K}$ fixed in Remark \ref{BI3} contains at least a morphism $f$ satisfying Lemma \ref{morphTid} above.

\begin{proof}

(of Theorem \ref{rank1}). Recall that, after Theorem \ref{transpi} and by Proposition 3.6 \cite{OV}, we only have to prove that there exists some $M_{\Psi}'$ such that, for all $i \geq M_{\Psi}'$, condition (\ref{condition}) holds for all $H \in Ob(\fusion_i^c) \setminus Ob(\linking_i)$. Actually we will prove the following:

\begin{itemize}

\item there exists some $M_{\Psi}'$ such that, for all $i \geq M_{\Psi}'$, $Ob(\fusion_i^c) = Ob(\linking_i)$.

\end{itemize}

Using the functor $\functor_i$, it is enough to prove that there exists such $M_{\Psi}$ such that, for all $S_i$-determined subgroups $R$, $R_i$ is $\fusion_i$-centric if and only if $R$ is $\fusion$-centric. Recall that Corollary \ref{corcentral1} proves the ``if'' implication in the above claim. Furthermore, Corollary \ref{corcentral2} says that if $R$ is not $S$-centric, then $R_i$ is not $S_i$-centric.

The rest of the proof is then devoted to show that if $R_i \leq S_i$ is an $S_i$-root such that $R = (R_i)^{\bullet}$ is $S$-centric but not $\fusion$-centric, then $R_i$ is not $\fusion_i$-centric. We can also assume that $R_i$ is maximal in the sense that if $Q_i \leq S_i$ is such that $R_i \lneqq Q_i$, then either $Q_i$ is $\fusion_i$-centric or it is not an $S_i$-root.

Let $H \in \conj{R}{S}$ be the representative of this $S$-conjugacy class fixed in $\pp'$ (Remark \ref{BI3}), and let $K \in \pp' \cap \conj{R}{\fusion}$ be fully $\fusion$-normalized. Note that both $H_i$ and $K_i$ are not $\fusion_i$-centric by assumption (Remark \ref{BI3}). Let $f \in \mm_{H,K}$ be as in Lemma \ref{morphTid}, and let
$$
\xymatrix@R=2mm@C=4mm{
 & & & & & & & & & & & \\
 & L_1 \ar[r]^{\gamma_1} & L_1 & & L_2 \ar[r]^{\gamma_2} & L_2 & & & & L_k \ar[r]^{\gamma_k} & L_k & \\
 & & & & & & & & & & & \\
R_0 \ar[rrr]_{f_1} \ar[ruu] & & & R_1 \ar[luu] \ar[ruu] \ar[rrr]_{f_2} & & & R_2 \ar[luu] \ar[r] & \ldots \ar[r] & R_{k-1} \ar[ruu] \ar[rrr]_{f_k} & & & R_k, \ar[luu], \\
}
$$
be the decomposition (\ref{Alperinlike}) fixed in Remark \ref{BI3} for the morphism $f$, together with the liftings $\varphi_j \in Aut_{\linking}(L_j)$. Let also $x \in N_S(H_i,R_i)$. By Lemma \ref{rootconjx}, $\Psi_i(x) x^{-1} \in C_T(R)$, or, equivalently, $\tau_0 = x^{-1} \Psi_i(x) \in C_T(H_i) = C_T(H)$.

We can now apply axiom (C) to $\varphi_1$ and the element $\tau_0$. By hypothesis (Lemma \ref{morphTid}), $f_1(C_T(H)) \leq T$, so in particular $f_1(\tau_0) = \tau_1$ for some $\tau_1 \in T$. Let then $t \in T$ be such that
$$
\tau_1 = t^{-1} \Psi_i(t),
$$
and let $Q_1 = tR_1t^{-1}$, $L'_1 = t^{-1} L_1 t$ and $\varphi_1' = \delta(t) \circ  \varphi_1 \circ \delta(x^{-1}) \in Aut_{\linking}(L'_1)$. It follows from Lemma \ref{detectmorph} that $\varphi_1'$ is $\Psi_i$-invariant, and $L_1'$ (or a certain proper subgroup) is $\fusion$-centric and $S_i$-determined.

Proceeding inductively through the whole sequence $f_1, \ldots, f_k$, we see that $\fusion_i$ contains a morphism sending $R_i$ to a subgroup $Q_i$ which is not $S_i$-centric, and hence $R_i$ is not $\fusion_i$-centric.

\end{proof}

\begin{rmk}

Since $p$-local compact groups of rank $1$ are approximated by $p$-local finite groups, we know (Theorem \ref{stable}) that the Stable Elements Theorem hold for all of them. This result was used in \cite{BLO2} to prove Theorem 6.3, which states that, given a $p$-local finite group $\g$, a finite group $Q$ and a homomorphism $\rho: Q \to S$ such that $\rho(Q)$ is fully centralized in $\fusion$, there is a homotopy equivalence
$$
|C_{\linking}(\rho(Q))|^{\wedge}_p \stackrel{\simeq} \longrightarrow Map(BQ, B\g)_{B\rho},
$$
where $C_{\linking}(\rho(Q))$ is the centralizer linking system defined in appendix \S A \cite{BLO2}.

The proof for this result in \cite{BLO2} used an induction step on the order of $S$ and on the ``size'' of $\linking$. However, since rank $0$ $p$-local compact groups are just $p$-local finite groups, we could use the same argument to prove the above statement for $p$-local compact groups of rank $1$, with some minor modifications. We have skipped it in this paper since it is not a general argument (it would only apply to $p$-local compact groups of rank $1$), and requires a rather long proof.

\end{rmk}


\subsection{The unitary groups $U(n)$} 

We prove now that the $p$-local compact groups induced by the compact Lie groups $U(n)$, $n \geq 1$, are approximated by $p$-local finite groups. As proved in Theorem 9.10 \cite{BLO3}, every compact Lie group $G$ gives rise to a $p$-local compact group $\g$ such that $(BG)^{\wedge}_p \simeq B\g$.
 
\begin{thm}\label{Un}

Let $\g(n)$ be the $p$-local compact group induced by the compact Lie group $U(n)$. Then every unstable Adams operation $\Psi$ approximates $\g(n)$ by $p$-local finite groups.

\end{thm}

The key point in proving this result is the particular isomorphism type of the Sylow subgroups of $U(n)$. Indeed, the Weyl group $W_n$ of a maximal torus of $U(n)$ is (isomorphic to) the symmetric group on $n$-letters, $\Sigma_n$. The action of $W_n$ on the maximal tori of $U(n)$ is easier to understand on the maximal torus of $U(n)$ formed by the diagonal matrices, $T$, where it acts by permuting the $n$ nontrivial entries of a diagonal matrix (see \S 3 \cite{Mimura-Toda} for further details). Furthermore, the following extension is split
$$
T \longrightarrow N_{U(n)}(T) \stackrel{\pi} \longrightarrow W_n.
$$

Let us fix some notation. Let $\{t_k\}_{k \geq 0}$ be a basis for $\ptor$, that is, each $t_k$ has order $p^k$, and $t_{k+1}^p = t_k$ for all $k$. Let also $T = (\ptor)^n$, with basis $\{(t_{k_1}^{(1)}, \ldots, t_{k_n}^{(n)})_{k_1, \ldots, k_n \geq 0}\}$. This way, the symmetric group $\Sigma_n$ acts on $T$ by permuting the superindexes. In addition,  if $\Sigma \in Syl_p(\Sigma_n)$, then $S = T \rtimes \Sigma$ can be identified with the Sylow $p$-subgroup of the $p$-local compact group $\g(n)$.

\begin{lmm}\label{centralizerUn}

Let $P \leq S$. Then, $C_T(P)$ is a discrete $p$-subtorus of $T$.

\end{lmm}

\begin{proof}

We proof that every element in $C_T(P)$ is infinitely $p$-divisible. Let $\pi = P/(P \cap T) \leq \Sigma \leq \Sigma_n$, and let $t \in C_T(P)$. Note that this means that $xtx^{-1} = t$ for all $x \in P$.

In the basis that we have fixed above, $t = (\lambda_1 t_{k_1}^{(1)}, \ldots, \lambda_n t_{k_n}^{(n)})$, where the coefficients $\lambda_j$ are in $(\Z/p)^\times$, and, for all $\sigma \in \pi$, if $\sigma(j) = l$, then $\lambda_j = \lambda_l$ and $k_j = k_l$.

For each orbit of the action of $\pi$ in the set $\{1, \ldots, n\}$ let $j$ be a representative. Let also $t_j = \lambda_j t_{k_j}^{(j)}$ be the $j$-th coordinate of $t$, and let $u_j$ be a $p$-th root of $t_j$ in $\ptor$.

We can consider the element $t' \in T$ which, in the coordinate $l$, has the $p$-th root $u_j$ which corresponds to the orbit of $l$ in $\{1, \ldots, n\}$ under the action of $\pi$. This element is then easily seen to be a $p$-th root of $t$, as well as invariant under the action of $\pi$. Thus, $t' \in C_T(P)$, and this proves that every element in $C_T(P)$ is infinitely $p$-divisible.

\end{proof}

\begin{proof}

(of Theorem \ref{Un}). We first prove the following statement:

\begin{itemize}

\item Let $P \leq S$. There exists some $M'_P \geq 0$ such that, for all $i \geq M'_P$, if $R \in \conj{P}{S}$ is $S_i$-determined, then $R_i$ is $S_i$-conjugate to $P_i$.

\end{itemize}

By Lemma \ref{rootconjx}, for all $y \in N_S(P_i,R_i)$ we have $y^{-1} \Psi_i(y) \in C_T(P_i)$. Also, since $S$ is $S_i$-determined, the subgroup $S_i$ contains representatives of all the elements in $\Sigma$, and hence we can assume that $y \in T$.

Consider now the map
$$
\xymatrix@R=1mm{
T \ar[rr]^{\Psi_i^{\ast}} & & T\\
t \ar@{|->}[rr] & & t^{-1} \Psi_i(t).\\
}
$$
Since $T$ is abelian this is a group homomorphism for all $i$, and in fact it is epi by the infinitely $p$-divisibility property of $T$. The kernel of $\Psi_i^{\ast}$ is the subgroup of fixed elements of $T$ under $\Psi_i$. Also, this morphism sends each cyclic subgroup of $T$ to itself.

It follows now by Lemma \ref{centralizerUn} that $y \in N_S(P_i, R_i)$ has the form $y = t_1 t_2$, where $t_1 \in T_i = Ker(\Psi_i^{\ast})$ and $t_2 \in C_T(P_i)$, and hence $R_i$ is $S_i$-conjugate to $P_i$.

In particular, since $\fusion^{\bullet}$ contains only finitely many $S$-conjugacy classes of non-$\fusion$-centric objects, and using the above claim, it is clear that there exists some $M$ such that, for all $i \geq M$ and each $S_i$-determined subgroup $R$, $R$ is $\fusion$-centric if and only if $R_i$ is $\fusion_i$-centric. Hence, by Proposition 3.6 \cite{OV}, it follows that for $i \geq M$ the fusion system $\fusion_i$ is saturated.

\end{proof}


The arguments to prove Theorem \ref{Un} do not apply to any of the other families of compact connected Lie groups. Note that $SO(3)$ and $SU(2)$ have already been considered in section \S 4 (although no explicit mention was made), and they are in fact examples of the complexity of the other families.



\end{document}